\documentclass[a4paper,10pt]{article}
\usepackage[utf8x]{inputenc}

\usepackage{fullpage}

\usepackage{amsmath,amssymb,amsthm,amsxtra,graphicx,subfigure}

\usepackage[colorlinks]{hyperref}

\numberwithin{equation}{section}

\newcommand{\R}{\mathbb{R}}

\newcommand{\eps}{\varepsilon}
\newcommand{\Laplace}{\Delta}

\DeclareMathOperator*{\argmin}{arg\,min}

\newcommand{\rd}{\mathrm{d}}

\newcommand{\abs}[1]{\left| #1 \right|}
\newcommand{\norm}[1]{\| #1 \|}

\newtheorem{theorem}{Theorem}[section]

\newtheorem{problem}[theorem]{Problem}
\newtheorem{proposition}[theorem]{Proposition}

\newtheorem{corollary}[theorem]{Corollary}
\newtheorem{remark}[theorem]{Remark}

\newtheorem{example}[theorem]{Example}

\newcommand\blfootnote[1]{%
  \begingroup
  \renewcommand\thefootnote{}\footnote{#1}%
  \addtocounter{footnote}{-1}%
  \endgroup
}

\title{Phase field methods for binary recovery }
\author{C. Brett \thanks{Mathematics Institute, University of Warwick, Coventry, CV4 7AL, UK (E-mail: {\tt c.brett@warwick.ac.uk}, {\tt a.s.dedner@warwick.ac.uk}, {\tt c.m.elliott@warwick.ac.uk})} \and A. S. Dedner\footnotemark[1] \and C. M. Elliott\footnotemark[1]}
\date{}

\begin{document}

\maketitle

\blfootnote{This work was supported by the UK Engineering and Physical Sciences Research Council (EPSRC) Grant EP/H023364/1.}

\begin{abstract}
We consider the inverse problem of recovering a binary function from blurred and noisy data. Such problems arise in many applications, for example image processing and optimal control of PDEs. Our formulation is based on the Mumford-Shah model, but with a phase field approximation to the perimeter regularisation. We use a double obstacle potential as well as a smooth double well potential. We introduce an iterative method for solving the problem, develop a suitable discretisation of this iterative method, and prove some convergence results. Numerical simulations are presented which illustrate the usefulness of the approach and the  relative merits of the phase field models.
\end{abstract}

\section{Introduction}

\label{sec:one}

A fundamental problem in the field of image processing is the following. Suppose we have a function $\bar{u}$ defined on a bounded and piecewise smooth domain $\Omega \subset \R^N$ for $N \leq 3$, which has been transformed by a linear operator $S$, and then corrupted by additive noise $\zeta$, such that we have data
\[
 y_d := S \bar{u} + \zeta.
\]
The problem is to recover $\bar{u}$ given $y_d$. Two immediate issues are that (a) $\zeta$ is unknown, so we will not be able to find $\bar{u}$ even with a good model for the space in which it lies (b) inverting $S$ may be ill-posed, so it will be difficult to find an approximation to $\bar{u}$ even if $\zeta=0$. 

We investigate this problem in the case that $\bar{u}$ is a binary function. We develop the theory with $S$ an abstract operator, but in examples we take $S$ to be the solution operator of an elliptic PDE. In this case the problem becomes one in PDE constrained optimal control.

Our approach to modelling the problem is to minimise an energy functional consisting of an $L^2$ fidelity term plus a phase field approximation to minimal perimeter regularisation. This can be thought of as a relaxation of the Mumford-Shah segmentation model. In our phase field approximation we use the Ginzburg-Landau functional with both the smooth double well and double obstacle potentials.


\subsection{Motivating examples}

\label{sec:mot}

First we give examples from both image processing and optimal control of PDEs motivating the study of this problem:
\begin{itemize}
\item \textbf{Image segmentation} - We can represent a barcode by a 1D function which takes the value -1 when the barcode is white and 1 when it is black. When a barcode is scanned by a barcode reader this function becomes blurred (due to scattering in the air) and noisy (due to measurement error and imperfections in the barcode). So the machine only sees a corrupted signal, but from this it needs to determine the scanned barcode.
\item \textbf{Elliptic source recovery} - Suppose we have noisy data of a quantity $y$, which is related to another quantity $\bar{u}$ by some physical law. For example, let $\bar{u}$ represent a heat source, then the long term temperature distribution $y$ may be related to $\bar{u}$ by the solution of an elliptic PDE. Our goal could be to find the heat source that produces a particular temperature distribution.
\end{itemize}

\subsection{Background material}

\label{sec:background}

For the above problems to be tractable we naturally require some knowledge of the form of the operator $S$ and the noise $\zeta$. We also usually assume a specific form of $\bar{u}$, as this influences the best model to use. For example, in the barcode problem we could assume that the function we are trying to recover is a binary function taking the values -1 and 1, and that the bars have a minimum width. Some sets of assumptions on $S$, $\zeta$ and $\bar{u}$ that are made in the literature are the following:
\begin{enumerate}
\item \emph{Denoising and deblurring} - $S$ is a blurring operator (maybe the identity), $\zeta$ is Gaussian noise, and $\bar{u}$ is a piecewise smooth function  (\cite{Rudin-Osher-Fatemi}, \cite{Chambolle-Lions}, \cite{Chan-Esedoglu}).
\item \emph{Segmentation} - $S$ is a blurring operator (maybe the identity), $\zeta$ is Gaussian noise, and $\bar{u}$ is binary function (\cite{Mumford}, \cite{Esedoglu2004}, \cite{Choksi2010}). These are the assumptions we make in this work.
\item \emph{Binary image restoration} - $S$ is the identity, we have `salt and pepper' noise, and $\bar{u}$ is a binary function (\cite{Chan-Esedoglu-Nikolova}). This kind of noise gives each point of a binary function a probability of switching to the other value, so the data $y_d$ is also binary.
\end{enumerate}

Note that the above sets of assumptions have been named using terminology from image processing. Although our problem can be thought of as either an image processing or PDE constrained optimal control problem depending on the choice of $S$, we found most of the relevant literature to be from the image processing community. This is unsurprising since image processing is one of the main  applications of binary recovery. We end up taking $S$ to be the solution operator of an elliptic PDE, but try to use neutral language which reflects that our problem arises in these two fields.  

For segmentation, which we focus on in this work, a large proportion of the literature modifies one of the following two models when formulating the problem of Section \ref{sec:one} mathematically. We now introduce these models so the reader can see how our approach fits with the existing literature.

\begin{itemize}
\item \textbf{Model 1 (Mumford-Shah).}
This model, which was introduced in \cite{Mumford}, looks for piecewise smooth functions that minimise an energy functional.

Let $\Omega_i$ be disjoint open subsets with piecewise smooth boundaries such that the closure of $\bigcup \Omega_i$ is $\Omega$. Let $u$ be a function that is differentiable on $\bigcup \Omega_i$, but which is allowed to be discontinuous across $\Gamma:= \bigcup \partial \Omega_i \setminus \partial \Omega$. Then the Mumford-Shah model involves minimising
\begin{equation}
\label{eqn:mum1}
E_1(u,\Gamma) = \frac{1}{2} \int_\Omega (u-y_d)^2 + \mu \int_{\Omega \setminus \Gamma} \abs{\nabla u}^2 + \sigma \abs{\Gamma},
\end{equation}
where $\abs{\Gamma}$ denotes the $N-1$ dimensional Hausdorff measure of $\Gamma$. The $\abs{\Gamma}$ term encourages minimising the length of the interface over which $u$ is discontinuous. 

If we restrict to minimising over binary functions that take the unknown value $a_i$ on $\Omega_i$ ($i=0,1$), then this energy functional becomes
\begin{equation*}
E_2(\{a_i\},\Gamma) = \frac{1}{2} \sum_i \int_{\Omega_i} (a_i-y_d)^2 + \sigma \abs{\Gamma}.
\end{equation*}
For fixed $\Gamma$ note that $E_2$ is minimised with respect to $\{a_i\}$ by setting
\[
a_i= \frac{1}{\abs{\Omega_i}} \int_{\Omega_i} y_d.
\]
So the problem reduces to just finding $\Gamma$, the locations of the discontinuities.

Due to the spaces of functions we are minimising over, both of the above variants of the Mumford-Shah model are nonconvex problems. In our work will use a relaxation of (\ref{eqn:mum1}) based on a phase field approximation.

\item \textbf{Model 2 (ROF).}
The ROF (Rudin-Osher-Fatemi) model of \cite{Rudin-Osher-Fatemi} involves solving the following constrained minimisation problem over a suitable space of functions:
\begin{align}
 &\text{Minimise} \abs{u}_{TV} \nonumber \\ 
 &\text{with} \int_\Omega u = \int_\Omega y_d \text{ and } \int_\Omega (u-y_d)^2 = s^2. \label{eqn:constraints}
\end{align}
The term $\abs{u}_{TV}$ represents the total variation of $u$, and it can be defined even if $u$ is not continuous; the total variation of a function $u \in L^1(\Omega)$ is
\[
\abs{u}_{TV} := \sup \{-\int_\Omega u \mathrm{div}(\phi) \, \mathrm{d}x : \phi \in C^1_c(\Omega,\R^N), \norm{\phi}_{L^\infty(\Omega)} \leq 1 \}.
\]
Sometimes the notation $\int_\Omega \abs{\nabla u}$ is used instead of $\abs{u}_{TV}$ to highlight that the total variation of $u$ is equal to this quantity when it is well defined. The first constraint in (\ref{eqn:constraints}) says that the noise has zero mean and the second that it has standard deviation $s$. 

$BV(\Omega, \R)$ is the subspace of functions in $L^1(\Omega)$ which have finite total variation. Minimising this model over $u \in BV(\Omega, \R)$ can be related to the following problem for some value of $\sigma$:
\begin{align}
\text{Minimise} \quad \frac{1}{2} \norm{u-y_d}_{L^2(\Omega)}^2 + \sigma \abs{u}_{TV} \text{ over } BV(\Omega,\R). \label{eqn:rof2}
\end{align}
Note that (\ref{eqn:rof2}) can be thought of as a relaxation of (\ref{eqn:mum1}) with $\mu=0$; we minimise over a larger space of functions in order to get a convex problem. 

If we restrict to minimising over binary functions then (\ref{eqn:rof2}) becomes similar to the Mumford-Shah model. Suppose $u$ only takes the known values $a_0<a_1$ (i.e.\ $u \in BV(\Omega, \{a_0, a_1\})$), then
\[
\abs{u}_{TV} = (a_1-a_0)\mathrm{Per}(\{u=a_1\}) = (a_1-a_0)\abs{\Gamma},
\]
where the perimeter function $\mathrm{Per}(\Sigma) := \int_\Omega \abs{\nabla \chi_\Sigma}$ and $\Gamma$ is the set over which $u$ is discontinuous. So for binary functions, total variation regularisation is equivalent to both perimeter regularisation and the interfacial length regularisation in the Mumford-Shah model. In fact  (\ref{eqn:mum1}) and (\ref{eqn:rof2}) become equivalent.
\end{itemize}

Suppose that in addition to $u \in BV(\Omega, \{a_0, a_1\})$ we have salt and pepper noise. Then the data is binary and both models reduce to the geometric problem
\[
\min_{\Sigma_u \subset \Omega} \abs{\Sigma_u \Delta \Sigma_d} + \sigma (a_1-a_0) \mathrm{Per}(\Sigma_u).
\]
Here $\Sigma_u$ and $\Sigma_d$ denote respectively the sets where the unknown $u$ and data $y_d$ take the value $a_1$, $\abs{\cdot}$ is now the N dimensional Hausdorff measure, and $\Sigma_u \Delta \Sigma_d$ is the symmetric difference between the sets. 

\subsection{Phase field model}
\label{sec:model}

We base our model on the Mumford-Shah model, but minimise over the space of functions $BV(\Omega, \{a_0,a_1\})$, and generalise it to include the blurring operator $S$, which we suppose is known a priori. So we have the following nonconvex model with a parameter $\sigma$, which we will shortly relax in a different way to (\ref{eqn:rof2}):
\begin{equation}
\label{eqn:tv}
\argmin_{u \in BV(\Omega, \{a_0,a_1\})}  F(u) := \frac{1}{2} \norm{Su - y_d}^2_{L^2(\Omega)} + \sigma \mathrm{Per}(\{u=a_1\}).
\end{equation}

We require $S:L^2(\Omega) \to L^2(\Omega)$ to be continuous, linear, and have the mean preservation property i.e. $S 1 = 1$ and hence $Sc=c$ for any constant function $c$. Continuity is helpful for proving existence of minimisers. Linearity and the mean preservation property allow us to recover a function $\bar{u}: \Omega \to \{a_0,a_1\}$ from data $y_d$ by recovering a function $\bar{u}: \Omega \to \{-1,1\}$ from a scaled and shifted copy of $y_d$, so long as $a_0$ and $a_1$ are known. We assume this to be the case and will therefore restrict our attention to $a_0=-1$ and $a_1=1$ from now onwards.

Some examples of forms $S$ could take are:
\begin{enumerate}
\item \emph{Solution operator of elliptic PDE} - Let $S u := y$, where $y$ solves the elliptic boundary value problem
\begin{equation}
\label{eqn:ghk}
\begin{aligned}
- \alpha \Laplace y + y &= u &&\text{ in } \Omega  \\
\frac{\partial y}{\partial \nu} &= 0  &&\text{ on } \partial \Omega.
\end{aligned}
\end{equation} 
For any $u \in L^2(\Omega)$ this equation has a unique weak solution $y \in H^1(\Omega)$ which satisfies the stability estimate
\begin{align}
\norm{y}_{L^2(\Omega)} = \norm{S u}_{L^2(\Omega)} \leq C_s(\alpha) \norm{u}_{L^2(\Omega)}, \label{eqn:bnd} 
\end{align}
where $C_s(\alpha) := \frac{1}{1+\alpha/C_p}$ and $C_p$ is the Poincar\'e constant. So $S$ has all the required properties. We also observe that evaluating $S$ is well-posed, but inverting $S$ is ill-posed, which motivates the need for our model. This is the operator we use for our numerics.
\item \emph{Convolution operator} - Let
\[
 S u := \phi_\alpha * u,
\]
where $\phi_\alpha$ is a suitable probability distribution of `size' $\alpha$, for example the Gaussian distribution
\[
\phi_\alpha(x) = \frac{1}{\alpha \sqrt{2 \pi}} \exp \left (-\frac{x^2}{2 \alpha^2} \right)
\]
of mean zero and variance $\alpha$, and $*$ is the convolution operation. Such an operator is used in the barcode problem of \cite{Esedoglu2004}, \cite{Choksi2010} and \cite{Choksi2011}.
\end{enumerate}

In both of these examples we have a parameter $\alpha$ which controls the extent of the blurring effect. Large $\alpha$ corresponds to heavy blurring and small $\alpha$ corresponds to light blurring. In our work the value of $\alpha$ is known a priori since we assume complete knowledge of $S$. However there are applications where we may want to relax this assumption, for example the barcode problem of \cite{Esedoglu2004}. In this application we do not know a-priori the distance of the barcode from the scanner, which means the level of blurring is unknown. This can be dealt with by fixing $\alpha$ to be some reasonable guess, or optimising for $\alpha$ at the same time as $u$.

We relax the model (\ref{eqn:tv}) by replacing the perimeter functional by the Ginzburg-Landau functional $G_\eps : L^1(\Omega) \to [0,\infty]$ defined by
\begin{equation*}
G_{\eps}(u) :=  \begin{cases}
\int_\Omega \frac{ \eps}{2} \abs{\nabla u}^2 + \frac{1}{\eps} \Psi(u) &u \in H^1(\Omega) \\
\infty &\text{otherwise}
\end{cases}
\end{equation*}
for some suitable $\Psi:\R \to \R$, and then minimising over $H^1(\Omega)$ instead of $BV(\Omega, \{-1,1\})$. So we consider
\begin{equation}
\argmin_{u \in H^1(\Omega)} F_\eps(u) := \frac{1}{2} \norm{Su -y_d}^2_{L^2(\Omega)} + \frac{\sigma}{c(\Psi)} \Bigg( \int_\Omega \frac{ \eps}{2} \abs{\nabla u}^2 + \frac{1}{ \eps} \Psi(u) \Bigg). 
\label{eqn:rough}
\end{equation}
We will focus on two different forms for the potential $\Psi$; the smooth double well potential
\begin{equation*}
\Psi_1(u) := \frac{1}{4} (1-u^2)^2,
\end{equation*}
and the double obstacle potential
\begin{align*}
\Psi_2(u) :=& \frac{1}{2}(1-u^2) + I_{[-1,1]} (u) \\
		=& \left\{ \begin{array}{cl}
   \frac{1}{2}(1-u^2) & \abs{u} \leq 1 \\
 \infty & \abs{u}>1 \end{array} \right..	
\end{align*} 

This approach, which is called a phase field approximation, results in a diffuse interface with minimisers no longer just taking the values $\{-1, 1 \}$, but values in the interval $[-1,1]$. It is still a nonconvex problem, but it has the advantage of allowing us to minimise over a smoother space of functions for which there is better developed theory. We are able to justify this approach with the following result.
\begin{theorem}
\label{thm:gamma}
Let $\Psi$ be the smooth double well potential $\Psi_1$. Then $G_\eps (u)$ $\Gamma$-converges in $L^1(\Omega)$ as $\eps \to 0$ to
\[\begin{cases}
c(\Psi_1) \mathrm{Per}(\{u=1\}) \quad &u \in BV(\Omega,\{-1,1\}) \\ \infty \quad &\text{otherwise}
\end{cases},\] 
where $c(\Psi_1)=2 \int_{-1}^1 \sqrt{2 \Psi_1(s)} \rd s =\frac{4 \sqrt{2} }{3}$.
\end{theorem}
\begin{proof}
See \cite{Modica1977}.
\end{proof}
 A similar result holds for the double obstacle potential, and performing a calculation we get that $c(\Psi_2)= \frac{\pi}{2}$ (see \cite{Curvature}). To simplify notation we let $\sigma_i=\sigma / c(\Psi_i)$. This ensures that the weighting given to the regularisation is asymptotically $\sigma$ for both potentials.

The different potentials lead to different formulations and we need to use different approaches to solve them. In particular, $\Psi_1$ leads to nonlinearity in the zeroth order terms, where as $\Psi_2$ causes nonlinearity by imposing constraints on the solution.

\subsection{Literature review}


We now mention other parts of the literature which overlap with aspects of this  work.

\textbf{Barcode problem.} The 1D version of our problem is related to the barcode problem of Esedoglu in \cite{Esedoglu2004}. This work was later extended by Choksi and Gennip in \cite{Choksi2010}. \cite{Choksi2011} uses similar ideas on QR barcodes. References for more general image processing literature can be found in Section \ref{sec:mot}.

\textbf{PDE constrained inverse problems.} A survey of the literature from the optimal control perspective can be found in \cite{Petra2011}. In addition, \cite{Tai2004} describes a number of applications where we want to recover piecewise constant functions, such as magnetic resonance imaging (MRI). The thesis \cite{Kepler2006} discusses a wide range of techniques for geometric inverse problems. \cite{Tai2007} recovers a piecewise constant diffusion coefficient from an elliptic PDE in 2D using the level set method.

\textbf{Phase field.}  In \cite{Kepler2006} there is a brief discussion of using a phase field approximation with the smooth double well potential for binary recovery. \cite{Esedoglu2004} and \cite{Choksi2010} use this idea for numerical simulations, though they do not justify the approach analytically. Theory for the phase field approximation with the double obstacle potential can be found in papers by Blowey and Elliott, including \cite{Elliott1992} and \cite{Elliott1992b}. In \cite{Sarbu2010} the double obstacle potential is used in the context of image processing, but without deblurring.

\textbf{Level set method.} This is an alternative way of recovering the discontinuities in our problem. It is discussed in \cite{Tai2004} and \cite{Tai2007}.

\textbf{Approximation of Mumford-Shah.} \cite{chambolle1999discrete} and related papers prove $\Gamma$-convergence results for finite element approximations of the Mumford-Shah functional. These results have some relation to the convergence results that we obtain using a different approach.

\vspace{0.1cm}
Our work differs from existing work, and hence offers a new contribution, in the following respects:
\begin{itemize}
\item We introduce the phase field approximation to the model right from the start (rather than at the last minute in order to allow numerical simulations). We therefore prove rigorous analytical results for this approximate model, which puts our approach on a much firmer footing than in existing work. 
\item Not only the smooth double well potential, but also the double obstacle potential is used for the phase field approximation. Results are proved for both simultaneously using an abstract framework.
\item We thoroughly investigate the dependency of the model on the parameters and perform a systematic comparison of the smooth double well potential and the double obstacle potential on a 1D problem. This highlights some advantages and attractive features of the latter in this setting.
\end{itemize}

\subsection{Layout}

In Section \ref{sec:gen} we introduce an abstract optimisation problem, an iterative method for finding critical points of this problem, and prove a convergence result for the iterative method. In Section \ref{sec:pfp} we show that (\ref{eqn:rough}) fits into this framework with both the smooth double well and double obstacle potentials. In Section \ref{chap:timedisc} we discuss a gradient flow formulation of (\ref{eqn:rough}) and its link to the iterative method. In Section \ref{chap:discretisation} we discretise the iterative method and prove another convergence result. We also look at a finite element discretisation for a particular choice of $S$. In Section \ref{sec:numerics} we demonstrate that implementations of the iterative method work well in 1 and 2 dimensions. In Section \ref{chap:comp} the performance of using both potentials is compared in detail for a 1D problem. In Appendix \ref{sec:appendix} we describe how we choose the parameters in our model for the numerics.

\section{Abstract framework}
\label{sec:gen}

Rather than developing separate theory for solving (\ref{eqn:rough}) with the smooth double well and obstacle potentials, it is advantageous to introduce an abstract framework that both problems fit into.

To this end let $V$ and $H$ be real Hilbert spaces with $V$ compactly embedded in $H$, and let $W$ be a closed convex nonempty subset of $V$. Let $b:V \times V \to \R$ and $c: H \times H \to \R$ be symmetric continuous bilinear forms with the properties
\begin{align*}
\exists \, \beta \text{ s.t. } b(\eta, \eta) \geq \beta \norm{\eta}_V^2 \quad \forall \eta \in V \\
c(\eta, \eta) \geq 0 \quad \forall \eta \in H.
\end{align*}
Let $l:V \to \R$ be a bounded linear functional and $J:V \to \R$ a continuous convex functional. 
With these objects we can define the energy functional $I:V \to \R$ by 
\begin{equation*}
I(\eta):=\frac{1}{2} b(\eta,\eta) +  J(\eta) - \frac{1}{2} c(\eta,\eta) - l(\eta),
\end{equation*}
which for positive constants $\alpha_0$ and $C_0$ we assume satisfies
\begin{equation}
\label{eqn:bndbelow}
I(\eta) \geq \alpha_0 \norm{\eta}_V^2 - C_0 \quad \forall \eta \in W.
\end{equation}
\begin{remark}
The functional $I$ can be decomposed in different ways into $b$, $J$, $c$ and $l$.
\end{remark}

\subsubsection*{Optimisation formulation}

Consider the following optimisation problem: Find $u \in W$ such that
\begin{equation}
\label{eqn:P}
I(u) = \inf_{\eta \in W} I(\eta). 
\end{equation}
We can show existence of a solution to (\ref{eqn:P}) with the following general result.

\begin{proposition}
\label{thm:abstractexist}
Let $A_1(\cdot):V \to \R$ be weakly lower semicontinuous and let $A_2(\cdot):H \to \R$ be continuous. If $A(\eta):=A_1(\eta)+A_2(\eta)$ is bounded below, then the following optimisation problem has a solution: Find $u \in W$ such that
\begin{equation*}
 A(u) = \inf_{\eta \in W} A(\eta).
\end{equation*}
\end{proposition}
\begin{proof}
This follows from standard theory; we construct an infimising sequence which we know is bounded in $V$, so have a subsequence which weakly converges to an element of $W$, and this element is a minimiser of $A$ by the properties of $A_1$ and $A_2$.
\end{proof}
\begin{corollary}
(\ref{eqn:P}) has a solution.  
\end{corollary}
\begin{proof}
Take $A_1(\eta) :=\frac{1}{2} b(\eta,\eta)+J(\eta)-l(\eta)$ and $A_2(\eta):=-\frac{1}{2} c(\eta,\eta)$. Recall that continuous convex functionals are weakly lower semicontinuous, so $A_1$ and $A_2$ satisfy the requirements of Theorem \ref{thm:abstractexist}. 
\end{proof}

Note that in general there is not a unique solution to (\ref{eqn:P}).

\subsubsection*{Variational inequality formulation}

By standard theory, solutions to (\ref{eqn:P}) must satisfy the following: Find $u \in W$ such that
\begin{equation}
\label{eqn:Q1}
b(u, \eta - u) + J(\eta) - J(u) \geq c(u, \eta-u) + l(\eta-u) \quad \forall \eta \in W. 
\end{equation}
Here we have used that $J$ is a convex function, so it has a subdifferential $\partial J$, which by definition satisfies
\[
J(\eta) - J(u) \geq \langle v,\eta-u \rangle \quad \forall v \in \partial J(u),
\]
where $\langle \cdot, \cdot \rangle$ denotes the duality pairing between $V^*$ and $V$. If $J$ is in addition G\^{a}teaux differentiable then (\ref{eqn:Q1}) is equivalent to the following variational inequality: Find $u \in W$ such that
\begin{equation}
\label{eqn:Q2}
b(u, \eta - u) + \langle J'(u),\eta-u \rangle \geq c(u, \eta-u) + l(\eta-u) \quad \forall \eta \in W. 
\end{equation} 
We often call solutions of (\ref{eqn:Q1}) critical points of (\ref{eqn:P}).
\begin{remark}
If $c(\eta,\eta) \leq \kappa b(\eta,\eta)$ for all $\eta \in V$ with $\kappa<1$, then (\ref{eqn:Q1}) has a unique solution. When we fit (\ref{eqn:rough}) into this framework, we find that this would require $\eps$ to be large. We intend to take $\eps$ small so that (\ref{eqn:rough}) approximates (\ref{eqn:tv}), which means we will not necessarily have uniqueness.
\end{remark}

Note that solutions of (\ref{eqn:P}) solve (\ref{eqn:Q1}), but the converse is not necessarily true. We nevertheless aim to solve (\ref{eqn:Q1}), as this is much easier in practice. Once a solution has been found, additional tests would have to  be used to verify that the solution is a local minimiser of $I$.

\subsubsection*{Iterative method}

We apply to (\ref{eqn:Q1}) the following generalisation of the iterative method of Barrett and Elliott \cite{Plasma}: Given $u^0 \in W$, for $n=1,2,...$ find $u^n \in W$ such that
\begin{equation}
\label{eqn:A1}
 b(u^n, \eta-u^n) + J(\eta) - J(u^{n}) \geq c(u^{n-1}, \eta-u^n) + l(\eta-u^n) \quad \forall \eta \in W. 
\end{equation}
If $J$ is in addition G\^{a}teaux differentiable then this is equivalent to the following iterative method: Given $u^0 \in W$, for $n=1,2,...$ find $u^n \in W$ such that
\begin{equation}
\label{eqn:A2a}
b(u^n, \eta-u^n) + \langle J'(u^n),\eta-u^n \rangle \geq c(u^{n-1}, \eta-u^n) + l(\eta-u^n) \quad \forall \eta \in W. 
\end{equation}
Note that $b(\eta,\eta)+J(\eta)$ is convex and $-c(\eta,\eta)-l(\eta)$ is concave. 

(\ref{eqn:A1}) and (\ref{eqn:A2a}) have unique solutions as they are equivalent to minimising a convex functional over $W$. Moreover we can prove the following convergence result.

\begin{theorem}
\label{thm:itss}
Every sequence $\{u^n\}$ generated by (\ref{eqn:A1}) satisfies 
\begin{equation}
\label{eqn:engineq}
I(u^n) + c(u^n-u^{n-1},u^n-u^{n-1}) + \beta \norm{u^n-u^{n-1}}_V^2 \leq I(u^{n-1})
\end{equation}
and has a subsequence which converges in $V$ to a critical point of (\ref{eqn:P}) i.e. a solution of (\ref{eqn:Q1}). Also, the limit of any subsequence of $\{u^n\}$ that converges weakly in $V$, and hence strongly in $H$, is a critical point of (\ref{eqn:P}). 
\end{theorem}
\begin{proof}
The proof is an extension to that of Theorem 6.1 in \cite{Plasma}. To deduce (\ref{eqn:engineq}) we test (\ref{eqn:A1}) with $\eta=u^{n-1}$ and use the coercivity of $b$. Because of the assumptions on $I$, $\{u^n\}$ is uniformly bounded in $V$, so we can extract a subsequence, which we also denote by $\{u^n\}$, that converges weakly in $V$ and strongly in $H$ to some element $u \in W$. The assumptions on $b$, $c$, $l$ and $J$ allow us to pass to the limit in (\ref{eqn:A1}) and deduce that $u$ satisfies (\ref{eqn:Q1}). The same argument applies to any subsequence, which proves the second part of the theorem.

To see why the convergence in the first part of the theorem is strong in $V$, note that now we know $u$ satisfies (\ref{eqn:Q1}), we can combine this inequality with (\ref{eqn:A1}) to get
\[
b(u-u^n,u-u^n) \leq c(u-u^{n-1},u-u^n).
\]
The result then follows using the coercivity of $b$ and the strong convergence of $u^n$ in $H$.
\end{proof}


\section{Binary recovery application}

\label{sec:pfp}

We now show that (\ref{eqn:rough}) with both the smooth double well and double obstacle potentials can be fitted into the framework of the previous section.

\subsubsection*{Smooth double well potential}

Set $V,W:=H^1(\Omega), H:=L^2(\Omega)$, let $S: H \to H$ satisfy the assumptions in Section \ref{sec:model}, and take
\begin{align*}
b(u,\eta) &:= (Su,S \eta) + \sigma_1 \eps (\nabla u, \nabla \eta) \\
c(u,\eta) &:= \frac{\sigma_1}{ \eps}(u, \eta) \\
l(u) &:= (S^*y_d, u) \\
J(u) &:= \frac{\sigma_1}{4 \eps} \int_\Omega u^4.
\end{align*}
Here and throughout this document $(\cdot,\cdot)$ denotes the $L^2(\Omega)$ inner product. $S^*$ denotes the adjoint operator of $S$, which is defined as follows: For real Hilbert spaces $U$, $V$ the adjoint operator of a continuous linear operator $A:U \to V$ is the operator $A^*:V \to U$ such that
\[
(Au,v)_V=(u,A^*v)_U \quad \forall u \in U, v \in V.
\]

The above objects have the properties required in Section \ref{sec:gen}. Coercivity of $b$ can be shown using a contradiction argument and that $S 0 = 0$. $J$ is well defined and continuous since $H^1(\Omega)$ is continuously embedded in $L^6(\Omega)$ for $\Omega \subset \R^N$ with $N \leq 3$. $I$ satisfies assumption (\ref{eqn:bndbelow}) since 
\[
\int_\Omega \frac{u^4}{4}-\frac{u^2}{2} \geq \int_\Omega \frac{u^2}{2}-1 = \frac{1}{2}\norm{u}_{L^2(\Omega)}^2- \abs{\Omega},
\] and so
\[
I(u)  \geq   \frac{\sigma_1 \eps}{2} \norm{\nabla u}^2_{L^2(\Omega)} + \frac{\sigma_1}{\eps}\norm{u}^2_{L^2(\Omega)} - \frac{\sigma_1}{\eps} \abs{\Omega} \geq \sigma_1 \min \Big \{ \frac{\eps}{2},\frac{1}{\eps} \Big \} \norm{u}_V^2 - \frac{\sigma_1}{\eps} \abs{\Omega} \quad \forall u \in W.
\]
Moreover $I$ equals $F_\eps$ from (\ref{eqn:rough}) with the smooth double well potential (up to an additive constant), so (\ref{eqn:P}) becomes: Given $y_d \in L^2(\Omega)$ find
\begin{equation}
\label{eqn:PF} 
\argmin_{
u \in H^1(\Omega)} F_1(u):=\frac{1}{2} \norm{Su - y_d}^2_{L^2(\Omega)} + \sigma_1 \Bigg( \int_\Omega \frac{ \eps}{2} \abs{\nabla u}^2 + \frac{1}{\eps} \Psi_1(u) \Bigg). 
\end{equation}

$J$ is G\^{a}teaux differentiable, so solutions to (\ref{eqn:PF}) satisfy (\ref{eqn:Q2}), which becomes: Given $y_d \in L^2(\Omega)$, find $u \in H^1(\Omega)$ such that
\[
(S^*(Su-y_d),\eta) + \sigma_1 \eps (\nabla u, \nabla \eta) + \frac{\sigma_1}{\eps}(u^3-u, \eta) = 0 \quad \forall \eta \in H^1(\Omega).
\]
In this example we have an equality instead of a variational inequality because $W$ is the full space $V$.

%
(\ref{eqn:A2a}) gives the following iterative method for solving the above variational inequality, and it converges by Theorem \ref{thm:itss}: Given $y_d \in L^2(\Omega)$ and $ u^0 \in H^1(\Omega)$, for $n=1,2,...$ find $u=u^n \in H^1(\Omega)$ such that
\begin{equation}
\label{eqn:dw1}
(S^*(Su-y_d), \eta) + \sigma_1 \eps (\nabla u, \nabla \eta)  + \frac{\sigma_1}{ \eps}(u^3-u^{n-1}, \eta)= 0  \quad \forall \eta \in H^1(\Omega). 
\end{equation}

\subsubsection*{Double obstacle potential}
\label{ex:doopt}

Define $K:=\{u \in H^1(\Omega) : \abs{u} \leq 1 \text{ a.e. in } \Omega\}$. Set $V:=H^1(\Omega),W:=K, H:=L^2(\Omega)$, let $S: H \to H$ satisfy the assumptions in Section \ref{sec:model}, and take
\begin{align*}
b(u,\eta) &:= (Su,S \eta) + \sigma_2 \eps (\nabla u, \nabla \eta) \\
c(u,\eta) &:= \frac{\sigma_2}{ \eps}(u, \eta) \\
l(u) &:= (S^*y_d, u) \\
J(u) &:= 0.
\end{align*}
The above objects have the properties required in Section \ref{sec:gen}. As with the smooth double well potential, $I$ satisfies assumption (\ref{eqn:bndbelow}) since for $u \in W$ we have
\[
-\int_\Omega \frac{u^2}{2} \geq \int_\Omega \frac{u^2}{2}-1 = \frac{1}{2}\norm{u}^2_{L^2(\Omega)} - \abs{\Omega}.
\]
Moreover $I$ equals $F_\eps$ from (\ref{eqn:rough}) with the double obstacle potential (up to an additive constant), so (\ref{eqn:P}) becomes: Given $y_d \in L^2(\Omega)$ find
\begin{equation} \label{eqn:PFO}
\argmin_{u \in K}  F_2(u) :=\frac{1}{2} \norm{Su - y_d}^2_{L^2(\Omega)} + \sigma_2 \Bigg( \int_\Omega \frac{ \eps}{2} \abs{\nabla u}^2 + \frac{1}{2\eps}  (1-u^2) \Bigg). 
\end{equation}

Solutions to ({\ref{eqn:PFO}) satisfy (\ref{eqn:Q2}), which becomes: Given $y_d \in L^2(\Omega)$, find $u \in K$ such that
\begin{align*}
(S^*(S u - y_d),\eta-u) + \sigma_2 \eps (\nabla u, \nabla \eta- \nabla u) - \frac{\sigma_2}{\eps}( u,\eta-u) \geq 0 \quad \forall \eta \in K.
\end{align*}

(\ref{eqn:A2a}) gives the following iterative method for solving the above variational inequality, which converges by Theorem \ref{thm:itss}: Given $y_d \in L^2(\Omega)$ and $ u^0 \in K$, for $n=1,2,...$ find $u=u^n \in K$ such that
\begin{equation}
\label{eqn:do1}
(S^*(Su-y_d), \eta-u) + \sigma_2 \eps (\nabla u, \nabla \eta- \nabla u) - \frac{\sigma_2}{\eps}(u^{n-1}, \eta-u) \geq 0 \quad \forall \eta \in K. 
\end{equation}

\subsection{Alternative iterative methods}

In (\ref{eqn:dw1}) and (\ref{eqn:do1}) the $S^*S$ term is taken implicitly, so we need to be able to invert the operator $S^*S-\sigma_i\eps \Laplace$ efficiently, otherwise these iterative methods will be too computationally expensive. In some cases this may be possible, for example if $S$ is the identity, but in general this is not the case.

As we remarked earlier, the definitions of $b$ and $c$ that make $I$ correspond to (\ref{eqn:PF}) and (\ref{eqn:PFO}) are not unique. For example we can set $b(u, \eta)= B(u,\eta) + \rho (u,\eta)$ and $c(u, \eta) = C(u,\eta) + \rho (u,\eta)$ for some $\rho \geq 0$. The $\rho (u,\eta)$ terms cancel out in $I$, so defining $B$ and $C$ the same way $b$ and $c$ were defined earlier in this section gives the same optimisation problems (\ref{eqn:PF}) and (\ref{eqn:PFO}). But the corresponding iterative methods are different. The point of this is that the $\rho(u,\eta)$ term is convex (when $\eta=u$), so it gives us more flexibility in how we define $B$ and $C$ while still having $b$ and $c$ satisfy the coercivity and positivity assumptions.

In particular, for suitably large $\rho$ we can take the $S^*S$ term explicitly (which in our framework corresponds to moving it from $b$ to $c$), and also take the $\frac{\sigma_i}{\eps}(u, \eta)$ term implicitly (i.e. move it from $c$ to $b$). So for our examples this corresponds to taking
\begin{align*}
b(u,\eta) &:= \rho(u, \eta) + \sigma_i \eps (\nabla u, \nabla \eta) - \frac{\sigma_i}{\eps}(u, \eta), \\ 
c(u,\eta) &:= \rho(u,\eta) - (S^*Su, \eta).
\end{align*}
A restriction such as $\rho > \max \{\frac{\sigma_i}{\eps},C_s^2 \}$, where $C_s$ is the stability constant from (\ref{eqn:bnd}), is then sufficient for both $b$ to be coercive and $c$ to be nonnegative. So we have the following iterative methods, which are in general easier to solve computationally than (\ref{eqn:dw1}) and (\ref{eqn:do1}). 
\begin{example}[Smooth double well] \label{ex:dwm2}
Given $y_d \in L^2(\Omega)$ and $ u^0 \in H^1(\Omega)$, for $n=1,2,...$ find $u=u^n \in H^1(\Omega)$ such that
\begin{equation}
\rho (u-u^{n-1}, \eta) + (S^*(Su^{n-1}-y_d), \eta) + \sigma_1 \eps (\nabla u, \nabla \eta) + \frac{\sigma_1}{\eps}(u^3-u, \eta) = 0 
\label{eqn:lala}
\end{equation}
for all $\eta \in H^1(\Omega)$. 
\end{example}

\begin{example}[Double obstacle] \label{ex:dom2}
Given $y_d \in L^2(\Omega)$ and $ u^0 \in K$, for $n=1,2,...$ find $u=u^n \in K$ such that
\begin{equation}
\rho (u-u^{n-1}, \eta-u) + (S^*(Su^{n-1}-y_d), \eta-u) + \sigma_2 \eps (\nabla u, \nabla \eta - \nabla u) - \frac{\sigma_2}{\eps}(u, \eta-u) \geq 0 
\label{eqn:do2}
\end{equation}
for all $\eta \in K$. 
\end{example}

 
When solving Example \ref{ex:dwm2} in practice, it is more convenient for us to solve a linear equation. Therefore we linearise the $J'(u)$ term in (\ref{eqn:lala}) and consider the following iterative method.
\begin{example}[Smooth double well] \label{ex:dwm3}
Given $y_d \in L^2(\Omega)$ and $ u^0 \in H^1(\Omega)$, for $n=1,2,...$ find $u=u^n \in H^1(\Omega)$ such that
\begin{equation} 
\rho (u-u^{n-1}, \eta) + (S^*(Su^{n-1}-y_d), \eta) + \sigma_1 \eps (\nabla u, \nabla \eta) + \frac{\sigma_1}{\eps}((u^{n-1})^2 u - u, \eta) = 0 
\label{eqn:lala2}
\end{equation}
for all $\eta \in H^1(\Omega)$. 
\end{example}
This iterative method lies outside of our framework, so the convergence theory does not necessarily hold. However it works well in practice.

To finish this section we show how we can reformulate the iterative methods to remove $S^*(Su^{n-1}-y_d)$ when $S$ is defined as in (\ref{eqn:ghk}). For example, (\ref{eqn:do2}) becomes the following.
\begin{example}[Double obstacle]
Given $y_d \in L^2(\Omega)$ and $ u^0 \in K$, for $n=1,2,...$ find $u=u^n \in K$ such that
\begin{equation*}
\rho (u-u^{n-1}, \eta-u) + (p^{n-1}, \eta-u) + \sigma_1 \eps (\nabla u, \nabla \eta- \nabla u) - \frac{\sigma_1}{\eps}(u, \eta-u) \geq 0 
\end{equation*}
for all $\eta \in K$, where $p^{n-1} \in H^1(\Omega)$ solves
\begin{equation*}
\begin{aligned}
\alpha (\nabla p^{n-1}, \nabla \eta) + (p^{n-1}, \eta) &= (y^{n-1}-y_d, \eta) \quad \forall \eta \in H^1(\Omega),
\end{aligned}
\end{equation*}
and $y^{n-1}$ solves the weak form of (\ref{eqn:ghk}) with $u=u^{n-1}$.
\end{example}


\section{Gradient flow}
\label{chap:timedisc}

In this section we investigate the gradient flow method for finding critical points of (\ref{eqn:PF}) and (\ref{eqn:PFO}) from an initial guess $u_0$. We prove that this method has some desirable properties, and note the link the to iterative method of the previous sections.


\subsubsection*{Smooth double well potential}

Let $u_0$ denote the initial guess of the solution and consider the $L^2$ gradient flow of $F_1$ in (\ref{eqn:PF}). 

\begin{problem}
\label{prob:dwgf}
Given $y_d \in L^2(\Omega)$ and $ u_0 \in H^1(\Omega)$, find $u \in L^2(0,T;H^1(\Omega))$ with weak time derivative $\partial_t u \in L^2(0,T;L^2(\Omega))$ such that $u(0) = u_0$ and
\begin{equation}
\label{eqn:jxx}
(\partial_t u(t), \eta) + (S^*(Su(t)-y_d), \eta) + \sigma_1 \eps (\nabla u(t), \nabla\eta) + \frac{\sigma_1}{\eps}(\Psi_1'(u(t)), \eta) = 0
\end{equation}
for all $\eta \in H^1(\Omega)$ and almost all $t \in (0,T)$.
\end{problem}

\begin{theorem}
Problem \ref{prob:dwgf} has a unique solution.
\end{theorem}
\begin{proof}
Note that Problem \ref{prob:dwgf} is very similar to the Allen-Cahn equation with the smooth double well potential, and the proof follows using standard techniques. See for example the references in Theorem \ref{thm:pew}, where existence and uniqueness is proved for smooth potentials in order to show existence and uniqueness for the double obstacle potential in the limit.
\end{proof}

\begin{theorem}
\label{thm:endec}
If $u$ is a sufficiently smooth solution of Problem \ref{prob:dwgf} then the energy $F_1(u(t))$ decreases over time.
\end{theorem}
\begin{proof}
For some $t \in (0,T)$ we can test (\ref{eqn:jxx}) with $\eta = \partial_t u(t)$ to get
\begin{equation}
\norm{\partial_t u(t)}_{L^2(\Omega)}^2 + (S^*(Su(t)-y_d), \partial_t u(t)) + \sigma_1 \eps (\nabla u(t), \nabla \partial_t u(t)) + \frac{\sigma_1}{\eps} (\Psi_1'(u(t)), \partial_t u(t)) = 0.
\label{eqn:earlier}
\end{equation}
Note that
\begin{align*}
(S^*(Su(t)-y_d),\partial_t u(t)) &= \frac{1}{2} \frac{\rd}{\rd t}  \norm{Su(t)-y_d}_{L^2(\Omega)}^2, \\
(\nabla u(t), \nabla \partial_t u(t)) &= \frac{1}{2} \frac{ \rd}{\rd t} \norm{\nabla u(t)}_{L^2(\Omega)}^2, \\ (\Psi_1'(u(t)), \partial_t u(t)) &= \frac{ \rd}{\rd t} \int_\Omega \Psi_1(u(t)),
\end{align*}
so equation (\ref{eqn:earlier}) is equivalent to
\[
\norm{\partial_t u(t)}_{L^2(\Omega)}^2 + \frac{\rd}{\rd t} \Big( \frac{1}{2} \norm{Su(t)-y_d}_{L^2(\Omega)}^2 + \frac{\sigma_1 \eps}{2} \norm{ \nabla u(t)}_{L^2(\Omega)}^2+ \frac{\sigma_1}{\eps} \int_\Omega \Psi_1(u(t)) \Big) = 0.
\]
Therefore as long as $\partial_t u(t)$ is not zero almost everywhere we have
\[
0 > -  \norm{\partial_t u(t)}_{L^2(\Omega)}^2 \geq \frac{\rd}{\rd t} F_1(u(t)),
\]
and hence the energy decreases.
\end{proof}

\subsubsection*{Double obstacle potential}

We can formulate a gradient flow for $F_2$ from (\ref{eqn:PFO}) in a similar way.

\begin{problem}
\label{prob:optf}
Given $y_d \in L^2(\Omega)$ and $ u_0 \in H^1(\Omega)$, find $u \in K_T$ with $\partial_t u \in L^2(0,T;L^2(\Omega))$ such that $u(0) = u_0$ and
\begin{equation}
\label{eqn:knb}
\begin{aligned}
(\partial_t u(t), \eta-u(t)) + (S^*(Su(t)-y_d), \eta-u(t)) &+ \sigma_2 \eps (\nabla u(t), \nabla\eta- \nabla u(t)) \\
 &- \frac{\sigma_2}{\eps}(u(t), \eta-u(t)) \geq 0
\end{aligned}
\end{equation}
for all $\eta \in K$ and almost all $t \in (0,T)$. Here 
\[
K_T := \{u \in L^2(0,T;H^1(\Omega)) : \abs{u} \leq 1 \text{ a.e. in }(0,T) \times  \Omega \}.
\]
\end{problem}

\begin{theorem}
\label{thm:pew}
Problem \ref{prob:optf} 
has a unique solution. Moreover, if $u$ is a sufficiently smooth solution then the energy $F_2(u(t))$ decreases over time.
\end{theorem}
\begin{proof}
This follows from a slight modification to the arguments for the double obstacle Allen-Cahn inequality in \cite{Blank2009, ChenElliott, Blank2009a, Curvature, Elliott1992b} to allow for the $S^*Su$ term.
\end{proof}

For both potentials it is important to consider whether $u(t)$ converges to a steady state as $t \to \infty$. These types of issues are investigated in \cite{Hale1988}, and in \cite{ChenElliott} for the 1D double obstacle potential. We do not discuss this as the focus of this work is on iterative methods.

\subsection{Link to iterative methods}

Particular first order discretisations in time of the gradient flow formulations are equivalent to the iterative methods of the previous section with $\rho = \frac{1}{\Delta t}$. But we only want to solve the optimisation problems (\ref{eqn:PF}) and (\ref{eqn:PFO}); we are not interested in the accuracy of solutions to (\ref{eqn:jxx}) and (\ref{eqn:knb}) at each point in time, but rather how well they approximate minimisers of $F_1$ and $F_2$ for large $t$. For this reason our method for solving (\ref{eqn:PF}) and (\ref{eqn:PFO}) should focus on decreasing the energy. The iterative methods of the previous sections are designed to have this property, where as discretisations in time of the gradient flows may not.

The scheme denoted by (\ref{eqn:A1}) of Barrett and Elliott motivated the convexity splitting implicit/explicit Euler scheme used in \cite{ElliottStuart1993}. See also \cite{Eyre1998}.

\section{Discretisation}

\label{chap:discretisation}

In this section we discretise the abstract iterative method of Section \ref{sec:gen} in space and analyse convergence of the discretisation. We then apply this theory to a finite element discretisation of (\ref{eqn:lala}) and (\ref{eqn:do2}) for $S$ defined by \ref{eqn:ghk}.

\subsection{Discrete abstract framework} 

Suppose we have a family of subspaces $V_h \subset V$ and closed convex nonempty subsets $W_h \subset V_h$ which approximate functions in $W$ increasingly well as some parameter $h \to 0$. In particular we suppose we have an approximation operator $P_h:W \to W_h$ such that
\begin{align} \label{eqn:interpol}
\norm{\eta - P_h\eta}_V \to 0 \text{ as } h \to 0 \quad \forall \eta \in W,
\end{align}
and that every sequence $\{\eta_h\} \subset W_h$ satisfies
\begin{equation} \label{eqn:xmxmx}
\eta_h \rightharpoonup \eta \text{ in } V \text{ as } h \to 0 \implies \eta \in W.
\end{equation}
\begin{remark} \label{rem:kkk}
Note that we do not require $W_h \subset W$. If this holds then (\ref{eqn:xmxmx}) follows automatically because $W$ is a closed convex subset of a Banach space, and hence is weakly sequentially closed.
\end{remark}

We now assume there exist objects $b_h$, $c_h$ and $l_h$ which satisfy the same assumptions as $b$, $c$ and $l$, with the boundedness and coercivity constants independent of $h$. We define
\begin{equation*}
I_h(\eta):=\frac{1}{2} b_h(\eta,\eta) +  J(\eta) - \frac{1}{2} c_h(\eta,\eta) - l_h(\eta),
\end{equation*}
and as in (\ref{eqn:bndbelow}) we assume that there exist positive constants $\alpha_1$ and $C_1$ independent of $h$ such that
\begin{equation}
\label{eqn:below}
I_h(\eta_h) \geq \alpha_1 \norm{\eta_h}_V^2 - C_1 \quad \forall \eta_h \in W_h.
\end{equation}
So minimisers of $I_h$ over $W_h$ (which exist, since $I_h$ satisfies the same assumptions as $I$) satisfy the following discrete problem: Find $u_h \in W_h$ such that
\begin{equation}
\label{eqn:Q1h}
b_h(u_h, \eta_h - u_h) + J(\eta_h) - J(u_h) \geq c_h(u_h, \eta_h-u_h) + l_h(\eta_h-u_h) \quad \forall \eta_h \in W_h. 
\end{equation}
If $J$ is in addition G\^{a}teaux differentiable then this is equivalent to the following discrete variational inequality: Find $u_h \in W_h$ such that
\begin{equation*}
b_h(u_h, \eta_h - u_h) + \langle J'(u_h),\eta_h-u_h \rangle \geq c_h(u_h, \eta_h-u_h) + l_h(\eta_h-u_h) \quad \forall \eta_h \in W_h. 
\end{equation*}

We need $b_h$, $c_h$ and $l_h$ to approximate their continuous counterparts as $h \to 0$. So we make the additional assumptions that for any bounded sequence $\{v_h\}\subset W$ we have
\begin{align}
&\norm{(b-b_h)(v_h,\cdot)}_{V^*} = \sup_{\eta_h \in V_h \setminus \{0\}} \frac{\abs{b(v_h,\eta_h)-b_h(v_h,\eta_h)}}{\norm{\eta_h}_V} \to 0, \label{eqn:fgfgf} \\
&\norm{(c-c_h)(v_h,\cdot)}_{V^*} = \sup_{\eta_h \in V_h \setminus \{0\}} \frac{\abs{c(v_h,\eta_h)-c_h(v_h,\eta_h)}}{\norm{\eta_h}_V} \to 0,  \nonumber\\
&\norm{l-l_h}_{V^*} = \sup_{\eta_h \in V_h \setminus \{0\}} \frac{\abs{l(\eta_h)-l_h(\eta_h)}}{\norm{\eta_h}_V } \to 0 \nonumber
\end{align}
as $h \to 0$. With these assumptions solutions of the discrete variational inequality (\ref{eqn:Q1h}) approximate solutions of the continuous variational inequality (\ref{eqn:Q1}) as $h \to 0$, as the following theorem shows.

\begin{theorem}
\label{thm:strang}
For any sequence $h_n \to 0$ the sequence $\{u_{h_n}\}$ of solutions to (\ref{eqn:Q1h}) has a subsequence which converges weakly in $V$, and hence strongly in $H$, to a critical point of (\ref{eqn:P}) i.e. a solution of  (\ref{eqn:Q1}). Moreover, the limit of any subsequence of $\{u_{h_n}\}$ that converges weakly in $V$, and hence strongly in $H$, is a critical point of (\ref{eqn:P}).
\end{theorem}
\begin{proof}
For a given $h$ we can find $u_h = \argmin_{\eta_h \in W_h } I_h(\eta_h)$, then for any $\eta_h \in W_h$, 
\[
I_h(u_h) \leq I_h(\eta_h) = \frac{1}{2} b_h(\eta_h,\eta_h) + J(\eta_h) - \frac{1}{2} c_h(\eta_h,\eta_h) - l_h(\eta_h).
\]
Fix $\eta \in W$ and set $\eta_h=P_h \eta \in W_h$. So $\{\eta_h\}$ is bounded in $V$ by (\ref{eqn:interpol}), which means $\abs{b_h(\eta_h,\eta_h) - b(\eta_h,\eta_h)} \leq \norm{(b_h-b)(\eta_h,\cdot)}_{V^*} \norm{\eta_h}_V \leq C$. Here and throughout this section $C$ denotes a generic constant independent of $h$ which may vary from line to line. A similar result holds for $l_h$, and $c_h$ is nonnegative, so
\[
I_h(u_h) \leq \frac{1}{2}b(\eta_h,\eta_h) + J(\eta_h) + \abs{l(\eta_h)}  + C.
\]
By the boundedness of $b$ and $l$,
\[
I_h(u_h) \leq C(\norm{\eta_h}_V^2 + J(\eta_h) + \norm{\eta_h}_V).
\]
Combining this with (\ref{eqn:below}) we get
\[
\norm{u_h}_V \leq C(\norm{\eta_h}_V + J(\eta_h) + 1).
\]
Now (\ref{eqn:interpol}) and the continuity of $J$ give that $J(\eta_h) \leq C$. In addition (\ref{eqn:interpol}) implies that for $h$ less than some $h_0$, $\norm{\eta_h}_V \leq \norm{\eta}_V+C$, and therefore $\norm{u_h}_V \leq C$.

From the above it follows that for any sequence $h_n \to 0$, $\{u_{h_n}\}$ is bounded in $V$. So we can find a subsequence, which we also denote by $\{u_{h_n}\}$, that converges weakly in $V$ and strongly in $H$ to some $u \in V$. In fact $u \in W$ by (\ref{eqn:xmxmx}). We now show that $u$ is a solution of (\ref{eqn:Q1}). 

Note that for all $\eta \in W$ we have
\begin{align*}
\liminf_{n \to \infty} b_{h_n}(u_{h_n},P_{h_n}\eta - u_{h_n}) &= \liminf_{n \to \infty} \Big ( b_{h_n}(u_{h_n},P_{h_n}\eta - u_{h_n}) \pm b(u_{h_n},P_{h_n}\eta - u_{h_n}) \pm b(u_{h_n},\eta - u_{h_n}) \Big ) \\
&= \liminf_{n \to \infty} \Big ( (b_{h_n}-b) (u_{h_n},P_{h_n} \eta - u_{h_n}) + b(u_{h_n},P_{h_n}\eta - \eta) + b(u_{h_n}, \eta-u_{h_n}) \Big ) \\
&= \liminf_{n \to \infty} b(u_{h_n}, \eta - u_{h_n}) \\
&\leq b(u,\eta-u).
\end{align*}
The final equality follows because $\lim_{n \to \infty} (b_{h_n}-b) (u_{h_n},P_{h_n} \eta - u_{h_n})=0$ by (\ref{eqn:fgfgf}) and $\lim_{n\to \infty} b(u_{h_n},P_{h_n}\eta - \eta) =0$ by (\ref{eqn:interpol}). The inequality follows from the lower semicontinuity of $b(\cdot,\cdot)$ and the continuity of $b(\cdot,\eta)$. Similar results hold for the $c_h$ and $l_h$ terms. This and the continuity and weak lower semicontinuity of $J$ gives
\begin{align*}
b(u,\eta-u) + J(\eta)-J(u) & \geq \liminf_{n \to \infty} \Big (b_{h_n}(u_{h_n},P_{h_n} \eta-u_{h_n}) + J(P_{h_n}\eta) - J(u_{h_n}) \Big ) \\
&\geq \liminf_{n \to \infty} \Big ( c_{h_n}(u_{h_n},P_{h_n} \eta- u_{h_n}) + l_{h_n} ( P_{h_n} \eta - u_{h_n}) \Big ) \\
&\geq c(u,\eta - u) + l(\eta - u) \quad \forall \eta \in W.          
\end{align*}
Hence $u$ is indeed a solution of (\ref{eqn:Q1}). 

The same argument applies to any weakly convergent subsequence, which proves the second part of the theorem.
\end{proof}

\begin{remark}
We could also assume we have functionals $J_h$ satisfying the same assumptions as $J$, with the continuity independent of $h$, plus the additional property that 
$v_{h_n} \rightharpoonup v$ in $W$ for $h_n \to 0$ implies $\liminf_{n \to \infty} J_{h_n}(v_{h_n}) \geq J(v)$. Then a proof almost identical to the above gives convergence for (\ref{eqn:Q1h}) with $J$ replaced by $J_h$. This allows numerical integration to be used on the $J$ term.
\end{remark}

As with (\ref{eqn:Q1}) in Section \ref{sec:gen}, we can consider an iterative method for solving (\ref{eqn:Q1h}): Given $u_h^0 \in W_h$, for $n=1,2,...$ find $u_h^n \in W_h$ such that
\begin{equation*}
b_h(u_h^n, \eta_h-u_h^n) +J(\eta_h) - J(u_h^n) \geq c_h(u_h^{n-1}, \eta_h-u_h^n) + l_h(\eta_h-u_h^n) \quad \forall \eta_h \in W_h. 
\end{equation*}
If $J$ is in addition G\^{a}teaux differentiable then this is equivalent to the following iterative method: Given $u_h^0 \in W_h$, for $n=1,2,...$ find $u_h^n \in W_h$ such that
\begin{equation*}
b_h(u_h^n, \eta_h-u_h^n) + (J'(u_h^n),\eta_h-u_h^n) \geq c_h(u_h^{n-1}, \eta_h-u_h^n) + l_h(\eta_h-u_h^n) \quad \forall \eta_h \in W_h. 
\end{equation*}

Since $b_h$, $c_h$ and $l_h$ satisfy the same assumptions as $b$, $c$, and $l$, the above iterative method still has the energy decreasing property, and we get convergence of iterates to a solution of (\ref{eqn:Q1h}). Then as $h \to 0$ the solutions of (\ref{eqn:Q1h}) converge to critical points of (\ref{eqn:P}) by Theorem \ref{thm:strang}. 

\subsection{Finite element discretisation of (\ref{eqn:lala}) and (\ref{eqn:do2})}
\label{sec:discmthds}

Assume that $\Omega$ is polyhedral and let $\{T_h\}$ be a family of uniform regular triangulations of $\Omega$ into disjoint open simplices with a maximal element size $h$. Associated with each $T_h$ we have the piecewise linear finite element space
\[
V_h := \{ v \in C^0(\bar{\Omega}) : v \! \mid_T \in P_1(T) \text{ for all } T \in T_h \} \subset H^1(\Omega),
\]
where $P_1(T)$ is the set of all linear affine functions on $T$. Also define
\[
K_h := \{v_h \in V_h : \abs{v_h} \leq 1 \text{ in } \Omega \}
\]
so that we have a finite element space analogous to $K$. Note that $K_h \subset K$ so Remark \ref{rem:kkk} applies. Take $P_h$ to be the operator that maps $u \in W$ to the unique $P_h u \in W_h$ such that
\[
(P_h u,\eta_h-u)_{H^1(\Omega)} \geq (u,\eta_h-u)_{H^1(\Omega)} \quad \forall \eta_h \in W_h.
\]
This operator satisfies equation \ref{eqn:interpol}, see e.g. Chapter 2 in \cite{RolandGlowinski1984}. 

Let $S$ be the solution operator of (\ref{eqn:ghk}), and denote by $S_h$ the discrete blurring operator. We intend this to approximate $S$, so we define $S_h$ to map $u \in L^2(\Omega)$ to the unique $y_h \in V_h$ satisfying
\begin{equation}
 \alpha (\nabla y_h, \nabla \eta_h) + (y_h, \eta_h) = (u,\eta_h) \quad \forall \eta_h \in V_h.
\label{eqn:jduy}
\end{equation}
A stability estimate the same as (\ref{eqn:bnd}) holds, so
\begin{equation}
\label{eqn:bnd2}
\norm{y_h}_{L^2(\Omega)} = \norm{S_h u}_{L^2(\Omega)} \leq C_s(\alpha) \norm{u}_{L^2(\Omega)},
\end{equation}
where as before $C_s(\alpha) = \frac{1}{1+\alpha/C_p}$. Also standard error analysis for elliptic PDEs says  
\begin{equation*}
\norm{y-y_h}_{L^2(\Omega)} \leq C h \norm{y}_{H^1(\Omega)},
\end{equation*}
which combined with (\ref{eqn:bnd2}) gives that
\begin{equation}
\label{eqn:erroranal}
\norm{(S-S_h)u}_{L^2(\Omega)} \leq C h \norm{Su}_{H^1(\Omega)} \leq C h \norm{u}_{L^2(\Omega)}.
\end{equation}

\begin{example}[Smooth double well]
\label{ex:sdw22}
Take the same definitions as in Example \ref{ex:dwm2}. In addition take $V_h$ as above, $W_h := V_h$, and define
\begin{align*}
b_h(u_h,\eta_h) &:= \rho (u_h,\eta_h) + \sigma_1 \eps(\nabla u_h, \nabla \eta_h) - \frac{\sigma_1}{ \eps}(u_h, \eta_h)  \\
c_h(u_h,\eta_h) &:= \rho (u_h,\eta_h) -(S_h u_h,S_h \eta_h) \\
l_h(u_h) &:= (S_h^* y_{d,h}, u_h) \\
J(u_h) &:= \frac{\sigma_1}{4\eps} \int_\Omega u_h^4,
\end{align*}
where $S_h$ is the discrete elliptic operator defined by (\ref{eqn:jduy}), and $y_{d,h}$ is the $L^2$-projection of $y_d$ onto $V_h$. 

For $\rho > \max \{\frac{\sigma_1}{\eps},C_s^2\}$, where $C_s$ is the stability constant from (\ref{eqn:bnd2}), all the assumptions of Theorem \ref{thm:itss} are satisfied, so we get the decreasing energy property and convergence of iterates for the following discrete iterative method: Given $y_{d,h}$, $u_h^0 \in V_h$, for $n=1,2,...$ find $u_h=u_h^n \in V_h$ such that
\begin{equation*}
\begin{aligned}
\rho(u_h-u_h^{n-1}, \eta_h) &+ (p_h^{n-1}, \eta_h) + \sigma_1 \eps (\nabla u_h, \nabla \eta_h) \\
&+ \frac{\sigma_1}{\eps}(u_h^3-u_h, \eta_h) = 0 \quad \forall \eta_h \in V_h,
\end{aligned}
\end{equation*}
where $y_h^{n-1}$, $p_h^{n-1} \in V_h$ satisfy
\begin{equation*}
\begin{aligned}
\alpha (\nabla y_h^{n-1}, \nabla \eta_h) + (y_h^{n-1}, \eta_h) &= (u_h^{n-1}, \eta_h) \\
\alpha (\nabla p_h^{n-1}, \nabla \eta_h) + (p_h^{n-1}, \eta_h) &= (y_h^{n-1}-y_{d,h}, \eta_h)
\end{aligned}
\end{equation*}
for all $\eta_h \in V_h$.

The assumptions of Theorem \ref{thm:strang} are also satisfied, since for a weakly convergent sequence $\{v_h\} \in V$ we have
\begin{align*}
\abs{(b_h-b)(v_h ,\eta_h)} & = \abs{(S_h v_h,S_h \eta_h)-(S v_h,S \eta_h)} \\
& \leq \abs{(S_h v_h, (S_h-S)\eta_h)} + \abs{((S_h-S) v_h, S\eta_h)}\\
&\leq \norm{S_hv_h}_{L^2(\Omega)}\norm{(S_h-S)\eta_h}_{L^2(\Omega)} + \norm{(S_h-S)v_h}_{L^2(\Omega)}\norm{S\eta_h}_{L^2(\Omega)}.
\end{align*}
Now using (\ref{eqn:bnd2}) and (\ref{eqn:erroranal}) we get
\[
\norm{(b_h-b)(v_h,\cdot)}_{H^1(\Omega)^*} \leq Ch \norm{v_h}_{H^1(\Omega)},
\]
and so $\norm{(b_h-b)(v_h,\cdot)}_{H^1(\Omega)^*} \to 0$ as $h \to 0$ by the boundedness of $\norm{v_h}_V$. Similar results hold for $c_h$ and $l_h$. Therefore we have convergence of limit points of the above discrete iterative method to critical points of (\ref{eqn:PF}) as $h \to 0$.
\end{example}

\begin{remark}
\label{rem:dwit}
As mentioned before Example \ref{ex:dwm3}, when solving the smooth double well problem in practice, we solve a finite element discretisation of the linearised iterative method (\ref{eqn:lala2}): Given $y_{d,h}$, $u_h^0 \in V_h$, for $n=1,2,...$ find $u_h=u_h^n \in V_h$ such that
\begin{equation}
\label{eqn:dvi2}
\begin{aligned}
\rho(u_h-u_h^{n-1}, \eta_h) &+ (p_h^{n-1}, \eta_h) + \sigma_1 \eps (\nabla u_h, \nabla \eta_h) \\
&+ \frac{\sigma_1}{\eps}((u_h^{n-1})^2 u_h - u_h, \eta_h) = 0 \quad \forall \eta_h \in V_h,
\end{aligned}
\end{equation}
where $y_h^{n-1}$, $p_h^{n-1} \in V_h$ satisfy
\begin{equation}
\label{eqn:adjointtt}
\begin{aligned}
\alpha (\nabla y_h^{n-1}, \nabla \eta_h) + (y_h^{n-1}, \eta_h) &= (u_h^{n-1}, \eta_h) \\
\alpha (\nabla p_h^{n-1}, \nabla \eta_h) + (p_h^{n-1}, \eta_h) &= (y_h^{n-1}-y_{d,h}, \eta_h)
\end{aligned}
\end{equation}
for all $\eta_h \in V_h$.

We use numerical integration on the linearised term. Note that the theorems do not necessarily hold for this iterative method, but it performs well in practice.
\end{remark}

\begin{example}[Double obstacle] \label{ex:do22}
Take the same definitions as in Example \ref{ex:dom2}. In addition take $V_h$ as above, $W_h:=K_h$, and define
\begin{align*}
b_h(u_h,\eta_h) &:= \rho (u_h,\eta_h) + \sigma_2 \eps(\nabla u_h, \nabla \eta_h) - \frac{\sigma_2}{ \eps}(u_h, \eta_h)  \\
c_h(u_h,\eta_h) &:= \rho (u_h,\eta_h) -(S_h u_h,S_h \eta_h) \\
l_h(u_h) &:= (S_h^* y_{d,h}, u_h) \\
J(u_h) &:= 0,
\end{align*}
where $S_h$ is the discrete elliptic operator defined by (\ref{eqn:jduy}), and $y_{d,h}$ is the $L^2$-projection of $y_d$ onto $V_h$. 

For $\rho > \max \{\frac{\sigma_2}{\eps},C_s^2\}$ all the assumptions of Theorem \ref{thm:itss} are satisfied, so we get the decreasing energy property and convergence of iterates for the following discrete iterative method: Given $y_{d,h}\in V_h$ and $u_h^0 \in K_h$, for $n=1,2,...$ find $u_h=u_h^n \in K_h$ such that
\begin{equation}
\label{eqn:dopot}
\begin{aligned}
\rho(u_h-u_h^{n-1},\eta_h-u_h) &+ (p_h^{n-1} , \eta_h-u_h) + \sigma_2 \eps (\nabla u_h, \nabla \eta_h-\nabla u_h) \\
&- \frac{\sigma_2}{\eps} (u_h,\eta_h-u_h) \geq 0 \quad \forall \eta_h \in K_h 
\end{aligned}
\end{equation}
where $y_h^{n-1}$, $p_h^{n-1} \in V_h$ satisfy
\begin{align*}
\alpha (\nabla y_h^{n-1}, \nabla \eta_h) + (y_h^{n-1}, \eta_h) &= (u_h^{n-1}, \eta_h) \\
\alpha (\nabla p_h^{n-1}, \nabla \eta_h) + (p_h^{n-1}, \eta_h) &= (y_h-y_{d,h},\eta_h)
\end{align*}
for all $\eta_h \in V_h$.

Theorem \ref{thm:strang} gives convergence of limit points of the above discrete iterative method to critical points of (\ref{eqn:PFO}) as $h \to 0$.
\end{example}

\subsection{Algorithms}

\label{sec:algo}

The discrete iterative methods in Examples \ref{ex:sdw22} and \ref{ex:do22} lead to the following algorithms for binary image recovery, which we implement and test in the next section.

\subsubsection*{Smooth double well potential}
\label{sec:bbbbb}

Given $y_{d,h} \in V_h$ and an initial guess $u_h^0 \in V_h$, set $n=1$ then:
\begin{enumerate}
\item Solve (\ref{eqn:adjointtt}) for $y_h^{n-1}$ then $p_h^{n-1}$;
\item Solve (\ref{eqn:dvi2}) for $u_h^n$;
\item If $\norm{u_h^{n} - u_h^{n-1}}_{L^2(\Omega)} < \text{TOL}$ terminate the algorithm. Else set $n=n+1$ and go to step 1;
\end{enumerate}

An alternative stopping criterion would be to wait until the change in energy $\abs{F_1(u_h^n)-F_1(u_h^{n-1})}$ is sufficiently small. This has the advantage that the energy decreasing result then guarantees our algorithm terminates. However the stopping criterion in the above algorithm also gives a strong indication of a steady state, and it seems to work better in practice.

Note that despite the blurring and noise, $y_{d,h}$ still contains a lot of information about the solution. Therefore it makes sense to scale and threshold $y_{d,h}$ in order to get a good initial guess for $u_h^0$. 

\subsubsection*{Double obstacle potential}

The algorithm for this potential is the same as for the smooth double well potential, but we instead solve the variational inequality (\ref{eqn:dopot}) in step 2.

One method for solving the variational inequalities at each iteration is the primal-dual active set (PDAS) method. It is applied to solving the variational inequalities arising in the Allen-Cahn inequality in \cite{Blank2009a}. We implemented this method and found it to work well. However, for the numerics in the next sections we use an alternative method known as the Truncated Nonsmooth Newton Multigrid (TNNMG) method (see \cite{CarstenThesis, Multigrid}), which performs very well.

\section{Numerics}
\label{sec:numerics}

In this section we show some numerical examples of binary recovery in 1 and 2 dimensions. The data is blurred by the solution operator of the elliptic PDE (\ref{eqn:ghk}), with the parameter $\alpha$ controlling the level of blurring. It also has additive Gaussian noise of mean zero and variance $\gamma$.
 
We do the recovery using the discrete iterative methods of Remark \ref{rem:dwit} (based on the smooth double well potential) and (\ref{eqn:dopot}) (based on the double obstacle potential). In practice we observe convergence of the full sequence of iterates to steady states, which are discrete critical points of (\ref{eqn:rough}). As we take $\eps$ and $h$ small, we believe that these critical points closely approximate a global minimiser of the model (\ref{eqn:tv}). This is because the iterative methods give us discrete critical points of the approximate model (\ref{eqn:rough}), which seem to be at least discrete local minimisers of (\ref{eqn:rough}), as different initial iterates and (valid) values of $\rho$ do not lead to different steady states. In addition, for small $\eps$ (and appropriate $h$) the critical points are close to being binary i.e. feasible minimisers of the model (\ref{eqn:tv}). We cannot be certain how close they really are to the global minimisers of (\ref{eqn:tv}) due to the lack of explicitly known global minimisers for interesting problems. Regardless, by artificially generating data from a known binary function, the numerical results show that for small $\eps$ (and appropriate $h$) our iterative methods are effective at recovering something close to the binary function.

The weighting given to the regularisation (the parameter $\sigma$), which defines the nonconvex model (\ref{eqn:tv}), is an important but challenging issue. If we take $\sigma$ too small then recovered functions still have artifacts of the noise. If $\sigma$ is too large then we loose some features we actually want to keep. We show some figures and discuss some results on the choice of $\sigma$ for related problems in Appendix \ref{sec:sig}, however the theory does not apply to our particular problem. In this section we just take values of $\sigma$ that we have experimentally determined to work well for the problem at hand.

For the implementation we use the Distributed and Unified Numerics Environment (DUNE), see \cite{ISTL,dunegridpaperI:08, dunegridpaperII:08, dune-web-page, dunefempaper:10, dunefem-web-page}. DUNE provides interfaces for grids, solvers and finite element spaces. Therefore once the algorithms are implemented, it takes minimal effort to change features of the implementation that would usually be fixed, such as the grid type, the dimension of the problem, and the type of finite elements used.

\subsection{1D numerics}
\label{sec:num1d}

The test problem in 1D is inspired by the barcode problem of \cite{Esedoglu2004}, which was mentioned as a motivating example in Section \ref{sec:mot}. We try to recover a binary function taking the values $\{-1,1\}$, which one can imagine represents a cross section of a barcode (with values of -1 corresponding to black parts of the barcode and values of 1 corresponding to white parts). We suppose this binary function is corrupted, giving blurred and noisy data that we want to decode. The main difference between our test problem and the barcode problem in \cite{Esedoglu2004} is that we have chosen blurring caused by the solution operator of an elliptic PDE instead of a convolution. Although this is not a realistic blurring operator specified by this application, if our approach is effective for this blurring operator then it is likely to be effective for other blurring operators.

The recovery using both the smooth double well and double obstacle potentials can be found in Figure \ref{fig:gaps}. The black lines represent the binary function that we want to recover, the blue lines are the artificial data we generate by adding blurring and noise, and the red lines are the recovered functions for each potential. Even by eye it is not clear exactly how many `bars' are in the binary functions, or the correct widths of the bars. But the recovered functions closely match the binary function we started with (which is why the black lines are almost hidden by the red lines), showing that our approach is effective. The figure also makes apparent one of the advantages of the double obstacle potential, which is that recovered functions take a form closer to what we actually want; binary functions. 

\begin{figure}[ht]
\centering
\subfigure[Smooth double well potential.]{
\includegraphics[width=0.7\textwidth, natwidth=768,natheight=225]{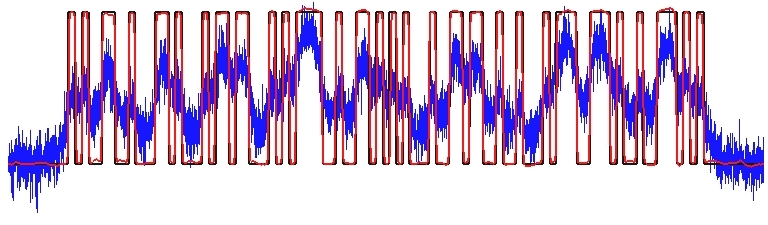}
\label{fig:gap1}
}
\subfigure[Double obstacle potential.]{
\includegraphics[width=0.7\textwidth,natwidth=768,natheight=225]{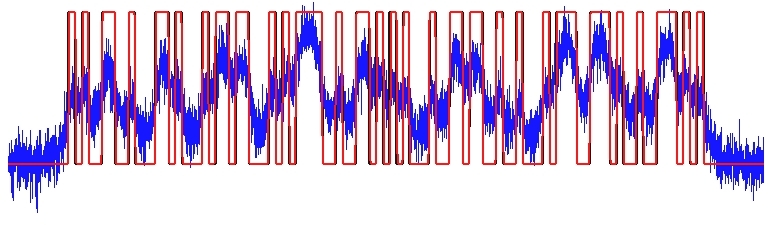}
\label{fig:gap2}
}
\caption{$\alpha=1e-4, \gamma=0.4, \sigma=1e-4, \eps=5.31e-4$ and $h=1.67e-4$.}
\label{fig:gaps}
\end{figure}

\subsection{2D numerics}
\label{sec:2dnum}

The test problems in 2D involve recovering binary functions with discontinuities of various shapes. In this dimension the problems have a natural interpretation as deblurring and denoising of images, but we also view them as binary source recovery problems for elliptic PDEs.

Figure \ref{fig:blobnoise} shows the recovery of a binary function using (\ref{eqn:rough}) with the smooth double well potential. The discontinuity is a `blob' shape and is marked by a black line. The blurred and noisy data for this function is shown in Figures \ref{fig:blobnoise0} and \ref{fig:blobnoise1}. Figure \ref{fig:blobnoise2} shows the recovered function, with a yellow line marking the zero level set. We can see that the yellow line closely matches the black line, except for a slight mismatch at the concave parts of the discontinuity. Note that we cannot make the interface as small as for the 1D problem as the resolution of the grid needed to resolve it makes this computationally expensive. Our implementation is capable of adaptivity, which lessens this cost somewhat, but we will not demonstrate this functionality in this work. With this simple visualisation the recovered function using the double obstacle potential looks very similar, so we do not include a figure of it. 

Figure \ref{fig:Anoise} (which can be interpreted in the same way as Figure \ref{fig:blobnoise}) shows the recovery of a binary function with a letter `A' shaped discontinuity. This time we use the double obstacle potential in (\ref{eqn:rough}), though the recovered function using the smooth double well potential looks similar. This example shows that the model can also recover discontinuities with corners reasonably accurately, but there is some rounding of these corners due to the regularisation.

To finish this section we show an example which relates to an application of binary image recovery in 2D. Figure \ref{fig:qr} shows the recovery of a binary function representing a QR code with 25x25 blocks (the size typically used to encode a URL). The yellow lines mark the discontinuity of the binary function. Figure~\ref{fig:qr1} shows the data with a red line marking the zero level set, and Figure \ref{fig:qr2} shows the recovered function. We see that features which are blurred below the zero level set (and which therefore would not be recovered by a simple projection) are nevertheless recovered by the model.

\begin{figure}[ht]
\centering
\subfigure[]{
\includegraphics[width=0.60\textwidth, natwidth=732,natheight=411]{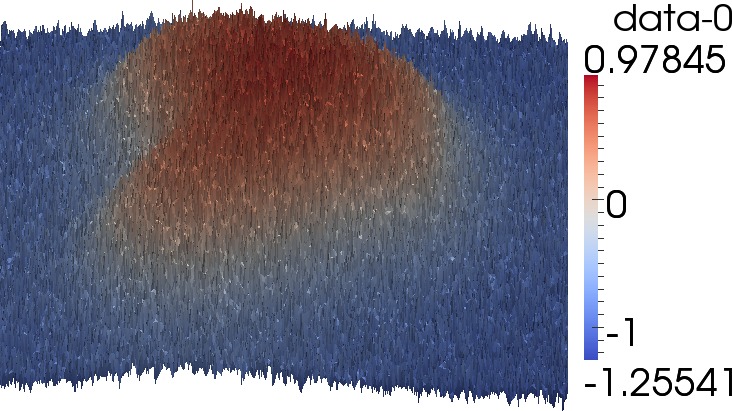}
\label{fig:blobnoise0}
}
\subfigure[]{
\includegraphics[width=0.47\textwidth, natwidth=730,natheight=567]{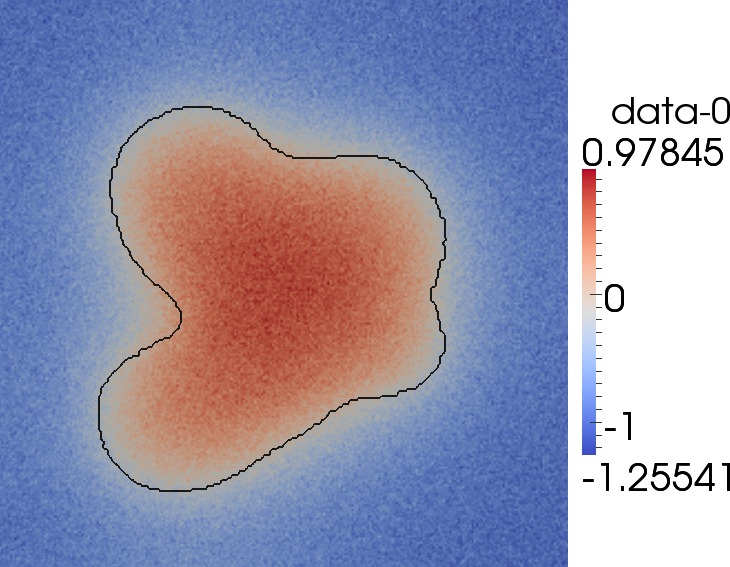}
\label{fig:blobnoise1}
}
\subfigure[]{
\includegraphics[width=0.47\textwidth, natwidth=753,natheight=567]{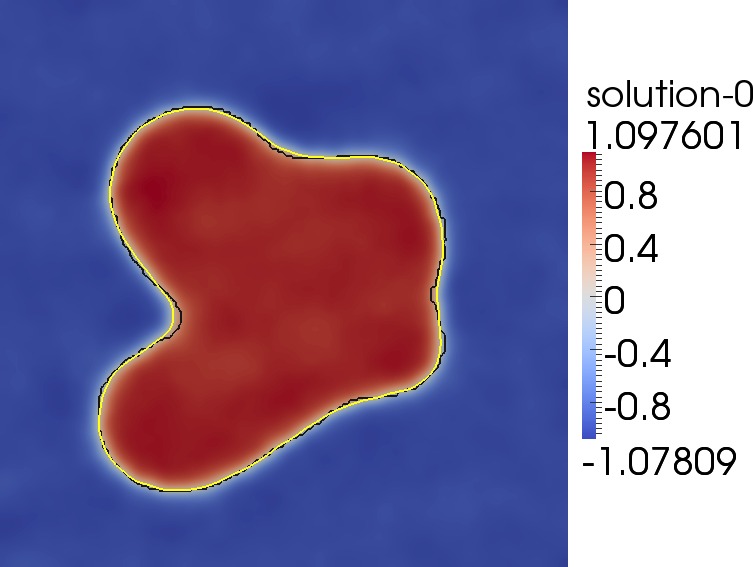}
\label{fig:blobnoise2}
}
\caption{$\alpha=0.01$, $\gamma = 0.2$, $\sigma = 1e-4$, $\eps = 0.00879$ and $h=0.00345$ using the smooth double well potential.}
\label{fig:blobnoise}
\end{figure}

\begin{figure}[ht]
\centering
\subfigure[]{
\includegraphics[width=0.6\textwidth,natwidth=742,natheight=394]{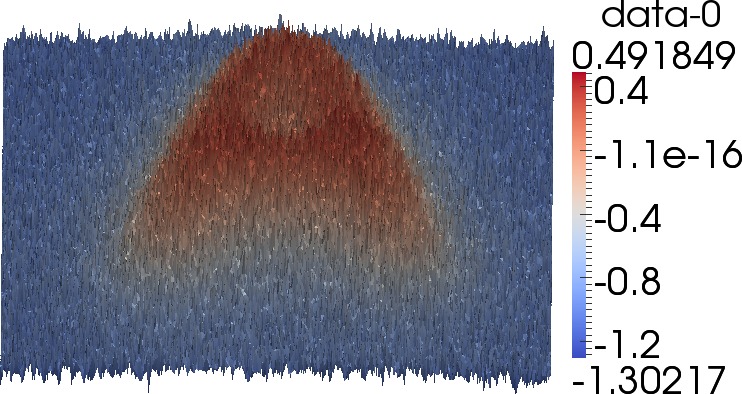}
\label{fig:Anoise0}
}
\subfigure[]{
\includegraphics[width=0.47\textwidth, natwidth=752,natheight=567]{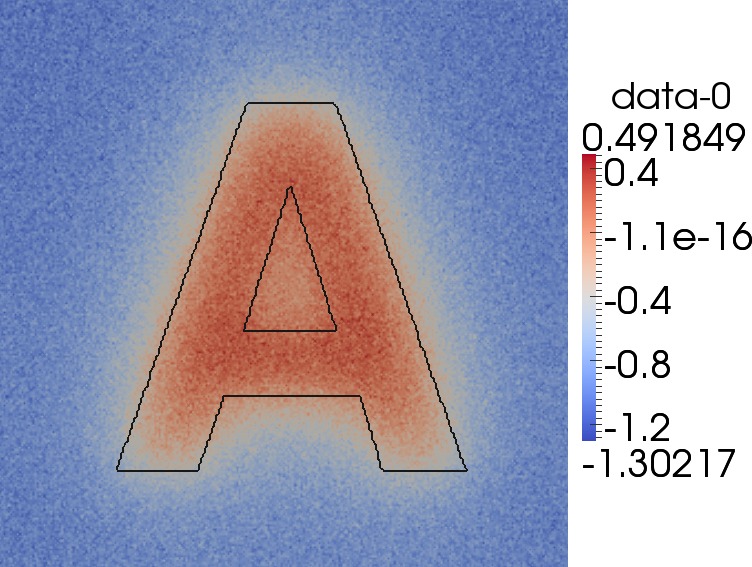}
\label{fig:Anoise1}
}
\subfigure[]{
\includegraphics[width=0.47\textwidth,natwidth=753,natheight=567]{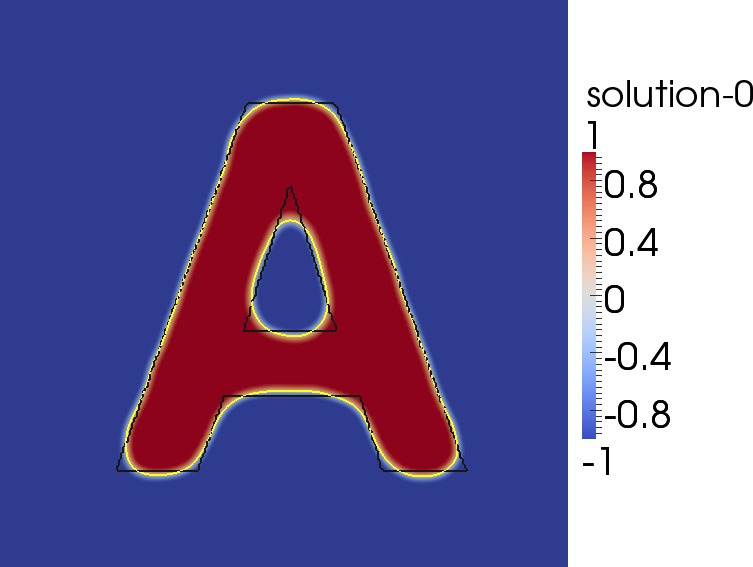}
\label{fig:Anoise2}
}
\caption{$\alpha=0.01$, $\gamma = 0.2$, $\sigma = 1e-4$, $\eps = 0.00879$ and $h=0.00345$ using the double obstacle potential.}
\label{fig:Anoise}
\end{figure}

\begin{figure}[ht]
\centering
\subfigure[]{
\includegraphics[width=0.47\textwidth,natwidth=683,natheight=684]{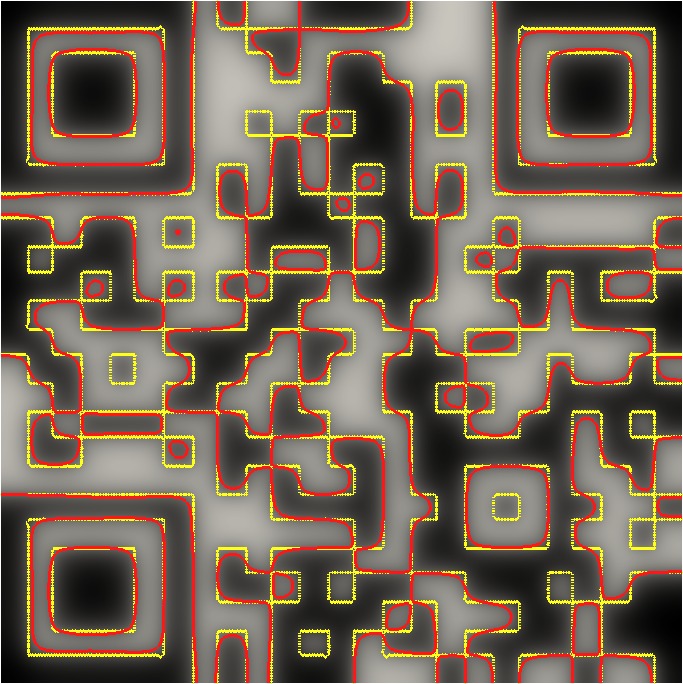}
\label{fig:qr1}
}
\subfigure[]{
\includegraphics[width=0.47\textwidth,natwidth=683,natheight=684]{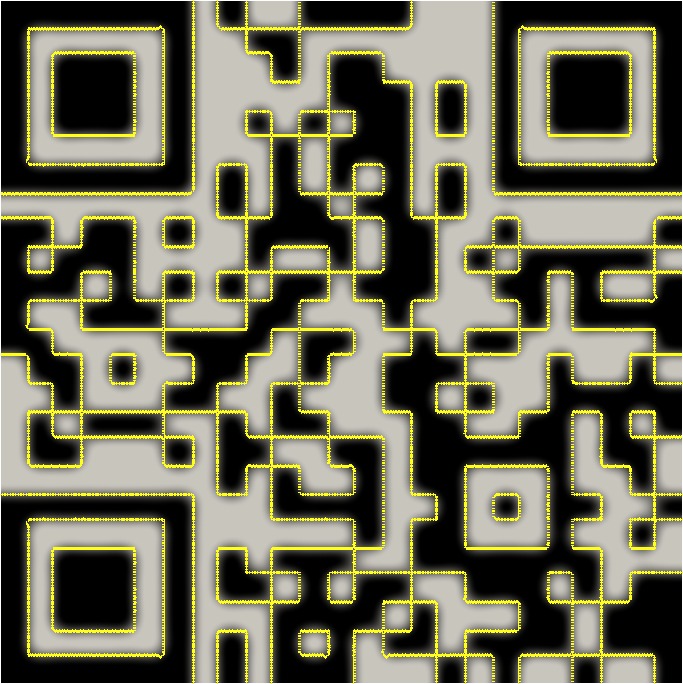}
\label{fig:qr2}
}
\caption{$\alpha=5e-4$, $\gamma = 0$, $\sigma = 1e-5$, $\eps = 0.00373$ and $h=0.00146$ using the double obstacle potential.}
\label{fig:qr}
\end{figure}

\section{Comparison of potentials in 1D}
\label{chap:comp}

Due to the $\Gamma$-convergence result of Theorem \ref{thm:gamma}, we expect that critical points of (\ref{eqn:rough}) for a given value of $\sigma$ using either the smooth double well or double obstacle potential will converge to critical point of (\ref{eqn:tv}) in the limit of small $\eps$. Of course the critical points they converge to are not guaranteed to be the same, but agreement of the limits is observed in practice, and for very small $\eps$ the recovered functions for both potentials are almost indistinguishable. However it is well known that for phase field type problems, the interface should be well resolved in order for an accurate spatial approximation. This means that the smaller $\eps$, the more grid points needed, and the higher the computational cost of the iterative methods. For many applications we only want to recover the location of the discontinuities in a binary function, which we suppose are given by the zero level set of the recovered function. 

This motivates us to consider in this section how well we can recover the locations of the discontinuities with $\eps$ of moderate size (rather than as small as possible), which is computationally cheaper. In this case the choice of potential does not just affect the implementation and speed of the iterative method; the recovered functions will in general look quite different, and there may be differences in how accurately or reliably the locations of the discontinuities are recovered. 

As in Section \ref{sec:numerics} we consider a problem with blurring caused by the solution operator of the elliptic PDE (\ref{eqn:ghk}) and additive Gaussian noise of mean zero and variance $\gamma$. We use the discrete iterative method of Remark \ref{rem:dwit} for the smooth double well potential and (\ref{eqn:dopot}) for the double obstacle potential.

At this stage it is helpful to recall the parameters we have introduced so far, as well as introduce a new parameter $\omega$, the width of the smallest bar in the binary function. The parameters are contained in Table \ref{tab:params}, and have been classified as follows:
\begin{itemize}
\item Problem parameters - Define the problem we are trying to solve. In applications we have no control over these, though we suppose they are known a priori.
\item Model parameters - Specify the model we will use to solve the problem. Different values can lead to the recovery of quite different functions, so they need to be chosen carefully.
\item Approximation parameters - We do not work with the model, but rather an approximation of it. These parameters control how good the approximation is.
\item Discretisation parameters - Affect the accuracy of the spatial discretisation in the iterative method.
\item Iteration parameters - Determine the behaviour of the iterative method.
\item Implementation parameters - Control the finer details of the implementation.
\end{itemize}

\begin{table}
\label{tab:params}
\begin{center}
    \begin{tabular}{ | l | l | l | l |}
    \hline
    Parameter & Description & Type of parameter & Optimal value \\
    \hline
    $\omega$ & Width of smallest bar in binary function& Problem & - \\
    $\alpha$ & Level of blurring & Problem & - \\
    $\gamma$ & Level of noise & Problem & - \\
    $\sigma$ & Weighting given to perimeter regularisation & Model & $\omega/80$ \\
    $\eps$ & Order of  width of interface & Approximation & $\omega / 4 \pi$\\
    $h$ & Grid width & Discretisation & $\omega / 32$ \\
    $u^0$ & Initial iterate & Iteration & - \\
    $\rho$ & Parameter in iterative method & Iteration & DW: $0.833$, DO: $0.588$\\
    TOL & Stopping criterion & Implementation & DW: $3e-4$, DO: $3.5e-4$ \\
    \hline
    \end{tabular}
\end{center}
\caption{Parameter types.}
\end{table}

We also have a number of less significant implementation parameters that handle the imprecision of computer arithmetic. These will be set to sensible values and ignored in our discussion. 

Motivated by the above discussion we now investigate differences between the smooth double well and double obstacle potentials in accuracy, reliability, speed, and implementational complexity. 

\subsection{Accuracy}
\label{sec:acc}

Denote the binary function we want to recover by $\bar{u}$ and the recovered function by $u_{\eps,h}$. We measure the accuracy of the recovery by calculating the error quantity
\[
E(u_{\eps,h}) := \frac{1}{4} \big |\abs{P(u_{\eps,h})}_{TV} - \abs{\bar{u}}_{TV} \big | + \frac{1}{2} \norm{P(u_{\eps,h})-\bar{u}}_{L^1(\Omega)},
\]
where $P$ is the $L^2$ projection onto the space $BV(\Omega, \{-1,1\})$ (i.e. $P(u_{\eps,h}) = 1$ when $u_{\eps,h} \geq 0$ and $-1$ when $u_{\eps,h}<0$). $\abs{u}_{TV}$ is the total variation of $u$, as defined in Section \ref{sec:background}. The integer part of $E(u_{\eps,h})$ tells us the absolute difference between the number of bars in the projected recovered function and $\bar{u}$. The decimal part tells us whether the discontinuities in the projected recovered function are in the correct locations. So $E$ measures the accuracy of the recovery in a sense that matters in applications.

We project because our best guess of $\bar{u}$ should lie in $BV(\Omega, \{-1,1\})$. The downside of this is that $P(u_{\eps,h})$ is not a minimiser of $(\ref{eqn:rough})$. It is important to note that the recovery using the double obstacle potential is naturally much closer to being binary than with the smooth double well potential, so projection is less necessary. This is a big advantage of using the double obstacle potential, which must be remembered when values of $E(u_{\eps,h})$ seem comparable.

The test problems we use for our comparison use the same binary function as in Section \ref{sec:num1d} (which has $\omega=\frac{1}{113}$), and different levels of blurring and noise i.e. a range of values of $\alpha$ and $\gamma$. We first fix $\sigma$ based on the size of $\omega$ (as described in Appendix \ref{sec:sig}) then choose good values of the approximation and discretisation parameters (as described in Appendices \ref{sec:eps} and \ref{sec:h}). So we have $\sigma=1e-4$, $\eps=7.06e-4$ and $h=2.77e-4$ for both potentials. Each realisation of the noise will be different, so we calculate an average $E$ over multiple realisations of the noise. As we observed earlier, we get the same steady state of (\ref{eqn:rough}) regardless of the choice of iteration parameters. The same is true for implementation parameters. So we ignore both these types of parameters in our discussion of accuracy.

We see in Figure \ref{fig:acc} that neither potential is the most accurate in all circumstances. For moderate levels of noise ($\gamma=0.2$), the double obstacle potential leads to a slightly more accurate recovery. However for high levels of noise ($\gamma=0.4$), the smooth double well potential seems to perform slightly better. Without projection the double obstacle potential always leads to a recovery which is significantly more accurate than the smooth double well potential.

\begin{figure}
\centering
\includegraphics[width=0.47\textwidth, natwidth=733,natheight=601]{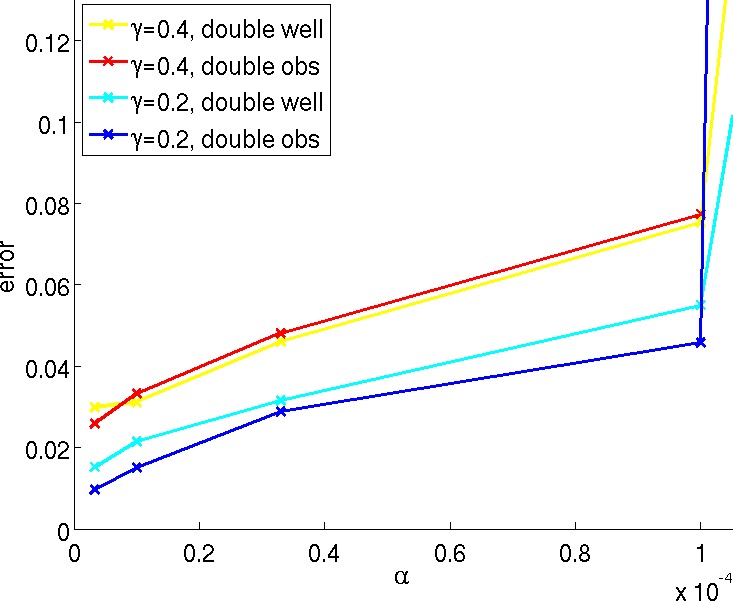}
\caption{The error (averaged over many realisations of the noise) for both potentials at different levels of blurring and noise.}
\label{fig:acc}
\end{figure}

\subsection{Reliability}
\label{sec:rel}

By reliability we refer to the range of problems (i.e. the levels of blurring and noise) over which a binary function can be recovered with reasonable accuracy; as the amount of blurring and noise are increased, eventually the recovered function does not resemble the binary function we wanted. Note that this range will depend on $\sigma$. We do not do a detailed comparison of reliability, but feel that it is comparable for both potentials. For example, we can see in Figure \ref{fig:acc} that $\alpha=1e-4$ and $\gamma=0.4$ is roughly the limit at which the correct number of bars can be recovered using either potential. 

\subsection{Speed}
\label{sec:speed}

The time it takes to recover a function which resembles the binary function is an important practical consideration. Where as accuracy is independent of the implementation, this is certainly not the case for speed. All but the inner workings of each iterative method in our implementation are identical, so we will do our best to make a fair comparison of speed.

We perform this comparison for the binary function of Section \ref{sec:num1d}, one choice of blurring and noise ($\alpha=1e-4$ and $\gamma=0.2$), and $\sigma$ as in Sections \ref{sec:acc} and \ref{sec:rel}. Choices of $\eps$ and $h$ as well as iteration and implementation parameters have a big impact on speed, so we will test two different combinations of these parameters. Our timings can be found in Table \ref{tab:runtimes}.

The runtimes for `accurate recovery' use $\eps$ and $h$ as in Sections \ref{sec:acc} and \ref{sec:rel}, and TOL as described in Appendix \ref{sec:stop}. These values have been chosen to ensure robustness. The table also contains timings for `rough recovery', where less conservative parameter values are used ($\eps=\frac{\omega}{2 \pi}$, $h = \frac{\omega}{20}$, and TOL as described in Appendix \ref{sec:stop}). For many problems we can still get a reasonable recovery with these parameter values, and it lowers the computation time significantly.

\begin{table}
 \centering
\begin{tabular}{|c|c|c|c|c|}
 \hline 
& Time for rough recovery (s)  & Time for accurate recovery (s) \\
\hline
\underline{Smooth double well} & & \\
Average time/it & 0.0359& 0.181 \\
\# iterations & 11 & 170 \\
Runtime  & 0.41& 29.9\\
\hline
\underline{Double obstacle} & & \\
Average time/it &0.0639 & 0.255 \\
\# iterations & 9 & 170 \\
Runtime & 0.58& 42.6\\
\hline
\end{tabular}
\caption{Average runtimes for with $\alpha=1e-4$ and $\gamma=0.2$.}
\label{tab:runtimes}
\end{table}

The recovery times are comparable for each potential for both rough and accurate recovery, though the smooth double well potential has a slight advantage for this size of problem. However we remark that the recovery time of the double obstacle potential scales better as the number of degrees of freedom in the discretisation increases, so it has better performance in 2D.

\subsection{Implementational complexity}
\label{sec:imp}

Implementing the iterative method for the double obstacle potential is less standard as we are solving variational inequality rather than a PDE. But it is no more complicated than implementing adaptivity, which is needed for the computational cost of the iterative method for the smooth double well potential to scale well to dimensions 2 and higher.

\subsection{Summary of comparison}

Both potentials can accurately recover binary functions over the same range of blurring and noise. If no projection is used, the double obstacle potential produces significantly more accurate results. Even with projection it is more accurate for moderate levels of blurring and noise. Our implementation using the smooth double well potential is slightly quicker for both accurate and rough binary recovery on our 1D test problem. However our implementation using the double obstacle potential, which is overall no more complicated, scales better to many degrees of freedom and so tends to be quicker in higher dimensions.




\appendix

\section{Parameter choices}

\label{sec:appendix}

In this appendix we describe our methodology for choosing parameter values for the  numerical tests and comparisons in Sections \ref{sec:numerics} and \ref{chap:comp}.  

\subsection{Choice of model parameter  \texorpdfstring{$\sigma$}{sigma}}
\label{sec:sig}

We recover different functions for different values of $\sigma$, so it is important to choose the `right' value. This is illustrated in Figure \ref{fig:sigmachoice}, where we show the recovered functions for the same problem as in Figure \ref{fig:gap1} for different values of $\sigma$. We see that $\sigma=5e-3$ leads to too few bars being recovered. The recovered function for $\sigma=1e-6$ follows the noise too much and does not resemble a binary function. With $\sigma=1e-4$ we recover something close to the binary function that generated the data, so we consider this to be a good value.

\begin{figure}[ht]
\centering
\subfigure[$\sigma=5e-3$.]{
\includegraphics[width=0.7\textwidth, natwidth=878,natheight=239]{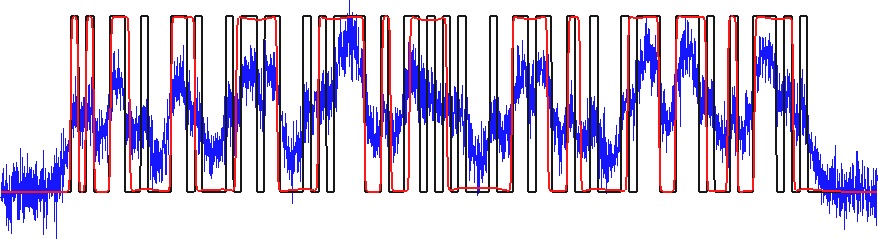}
}
\subfigure[$\sigma=1e-4$.]{
\includegraphics[width=0.7\textwidth, natwidth=878,natheight=239]{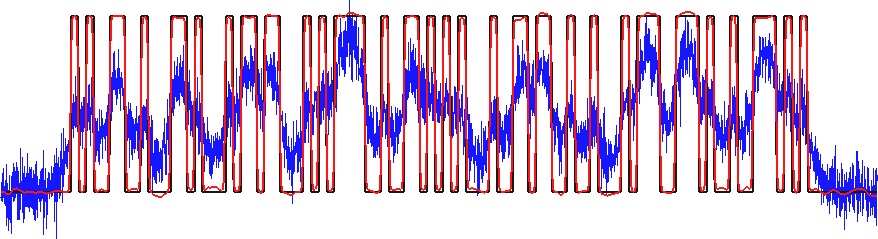}
}
\subfigure[$\sigma=1e-6$.]{
\includegraphics[width=0.7\textwidth, natwidth=878,natheight=266]{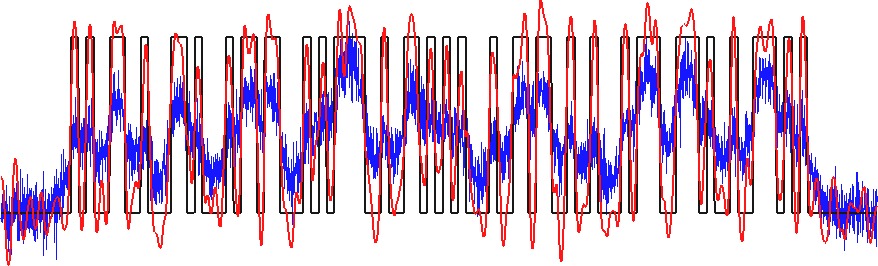}
}
\caption{The problem of Figure \ref{fig:gap1} with different values of $\sigma$.}
\label{fig:sigmachoice}
\end{figure}

It is known that the choice of $\sigma$ in (\ref{eqn:tv}) should be related to the variance of the noise. Noise with a large variance requires a large $\sigma$ in order for good recovery. We could try to figure out the variance of the noise from the data and use this to choose $\sigma$, however there is not an explicit form for the relationship. Instead we choose $\sigma$ based on the length scale of the features that we want to recover (i.e.\ the parameter $\omega$), and use the same $\sigma$ for all levels of noise. In applications this is generally known a priori e.g. for barcode recovery. This approach works well because we take $\sigma$ be as large as possible while not removing the features we want to recover, and hence perform the maximum amount of denoising. We do not seem to pay a significant price for this large $\sigma$ in cases where the noise is small, and this approach leads to a simple rule for choosing $\sigma$. The literature that gives us a heuristic way of choosing such a $\sigma$ is introduced below.

The following result shows that it is unwise to take $\sigma$ too large.
\begin{proposition}
There exists a $\sigma^*>0$ such that the minimiser of (\ref{eqn:tv}) is 0 iff $\sigma > \sigma^*$.
\end{proposition}
\begin{proof}
Proposition 5.7 in \cite{Chan-Esedoglu}.
\end{proof}

\noindent But we also need to be careful not to take $\sigma$ too small. In fact, since $S$ is known we have the following result in the 1D case.
\begin{theorem}
In the absence of noise there exists a $\sigma_*>0$ such that the minimiser of (\ref{eqn:tv}) is $\bar{u}$ whenever $\sigma \leq \sigma_*$. 
\end{theorem}
\begin{proof}
Proposition 5 in \cite{Esedoglu2004}.
\end{proof}

Another interesting result is Theorem 1.1 part 2 in \cite{Choksi2010}, which proves more explicit conditions on $\sigma$ to ensure exact recovery in the case that $S$ is a convolution with a hat function in 1D. Due to our complicated form for $S$ we are forced to use a more heuristic argument to choose a good value for $\sigma$. 

\cite{Chan-Esedoglu-Nikolova} shows that for the 1D case in the absence of blurring and noise (i.e. binary data), local and global minimisers of (\ref{eqn:tv}) can be calculated explicitly for a given value of $\sigma$. These considerations suggest we should take $\sigma$ to be smaller than a quarter of the size of the smallest object we want to recover. In particular, $\sigma=\frac{\omega}{8}$ seems like a sensible choice. But this assumes binary data. We have blurring, which means the differences between the functions in the $\norm{Su-y_d}_{L^2(\Omega)}^2$ term can be much smaller. Hence we take $\sigma$ an order of magnitude smaller i.e. $\sigma = \frac{\omega}{80}$. This $\sigma$ is still larger than the length scale of the noise (which is of order $h$), so the results in \cite{Chan-Esedoglu-Nikolova} say it will be removed. Numerical experiments confirm that this choice of $\sigma$ works well in practice. 

\subsection{Choice of \texorpdfstring{$\eps$}{eps}}

\label{sec:eps}

The phase field approximation in (\ref{eqn:rough}) results in solutions with interfaces of width $o(\eps)$. In order for an accurate spatial approximation we need a reasonable number of grid points across the interfaces. So a smaller $\eps$ requires more grid points and a higher computational cost. With this is mind we want to take $\eps$ as large as we can while still resolving the finest features of the binary function. So the choice of $\eps$ should be related to the value of $\omega$.

We assume that there is a linear relationship between the optimal choice of $\eps$ and $\omega$ and deduce the constant of proportionality $c_1$ such that we get a good recovery with $\pi \eps = c_1 \omega$. Note that $\pi \eps$ is the asymptotic width of the interface for minimisers of the Ginzburg-Landau functional with the double obstacle potential, and a good approximation with the smooth double well potential. The width of interfaces in minimisers of (\ref{eqn:rough}), a perturbed Ginzburg-Landau functional, are approximately the same size. So $c_1$ can be thought of as the relative width of the interface compared to the width of the smallest bar. 

To determine $c_1$ we recover a simple binary function which can be seen in Figure \ref{fig:testprob}. We take $\omega_1=\omega_2=\omega_3=0.2$ (i.e. bars of equal widths), as we found the case where all bars are at the finest length scale to be the hardest for accurate recovery. We consider different levels of blurring and noise and compute the error $E$ of the recovered functions. We take $\sigma$ to be the optimal value of $\frac{\omega}{80}$ that we decided upon in Appendix \ref{sec:sig}, and take $\pi \eps=50h$ to ensure that effects of the spatial discretisation do not distort our results.

\begin{figure}
\centering
\includegraphics[width=0.35\textwidth,natwidth=440,natheight=169]{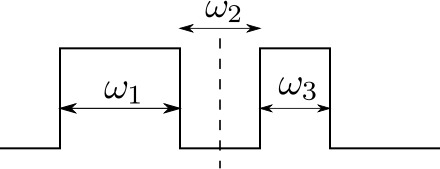}
\caption{A simple binary function.}
\label{fig:testprob}
\end{figure}

We observe that for a high signal to noise ratio we can take $c_1$ very large and still get accurate recovery ($\alpha=0.01$ in Figure \ref{fig:epsfac}), even though the bars do not separate properly (see Figure \ref{fig:gapss1}). For low signal to noise ratios ($\alpha=0.1$ in Figure \ref{fig:epsfac}) we need to take $c_1 \leq 0.5$ for accurate recovery, though it is not until $c \leq 0.25$ that the interfaces start to look reasonably sharp (see Figure \ref{fig:gapss2}). As expected there is not an accuracy penalty for taking $c_1$ too small, however it increases computation time by forcing us to take smaller $h$ in order to resolve the interfaces. This motivates us to take $c_1=0.25$ i.e. $\pi \eps= \frac{\omega}{4}$.

\begin{figure}
\centering
\includegraphics[width=0.47\textwidth,natwidth=725,natheight=599]{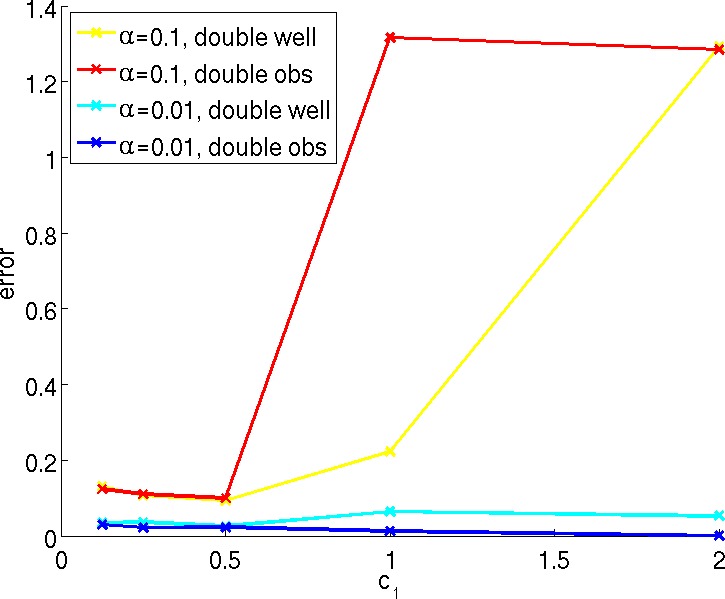}
\caption{Errors (averaged over many realisations of the noise) for both potentials at different levels of blurring and $\gamma=0.2$.}
\label{fig:epsfac}
\end{figure}

\begin{figure}[ht]
\centering
\subfigure[$c_1=2.0$.]{
\includegraphics[width=0.47\textwidth,natwidth=528,natheight=98]{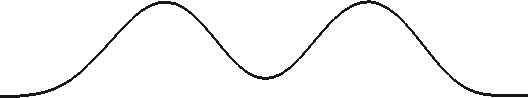}
\label{fig:gapss1}
}
\subfigure[$c_1=0.25$.]{
\includegraphics[width=0.47\textwidth,natwidth=528,natheight=110]{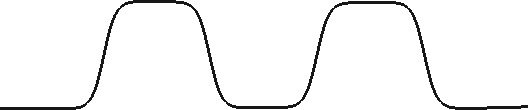}
\label{fig:gapss2}
}
\caption{The interfaces using the smooth double well potential with different values of $c_1$. \ref{fig:gapss2} shows the interfaces for $c_1=0.25$, which we decide is the optimal parameter value.}
\label{fig:gapss}
\end{figure}

\subsection{Choice of \texorpdfstring{$h$}{h}}
\label{sec:h}

We use the same test problems as in Appendix \ref{sec:eps} to deduce a constant factor $c_2$ such that we get a good recovery with $\pi \eps=c_2h$. Hence $c_2$ can be thought of as the number of grid elements across each interface. 

With a high signal to noise ratio ($\alpha=0.01$ in Figure \ref{fig:wfac}) it can actually be advantageous to have few grid points across the interface. In this case the recovered function would have to deviate a long way from the binary function in order for the projection to take an incorrect value on even a single grid point, and the data does not force sufficient deviation. As a result we can actually get perfect recovery on coarse grids. However, if we have a poorly resolved interface we are not well approximating our model and we may get a bad recovery for low signal to noise ratios ($\alpha=0.1$ in Figure \ref{fig:wfac}). 

We do not want to adjust the relationship between $\eps$ and $h$ for different levels of blurring and noise; we want a relationship for each potential that always works. This means we must properly resolve the interfaces. Figure \ref{fig:wfac} suggests that we can take $c_2=5$ for both potentials, however this leads to slightly jagged interfaces. Therefore we will again favour robustness and choose $c_2=8$ i.e. $\pi \eps=8h$.

\begin{figure}
\centering
\includegraphics[width=0.47\textwidth,natwidth=731,natheight=599]{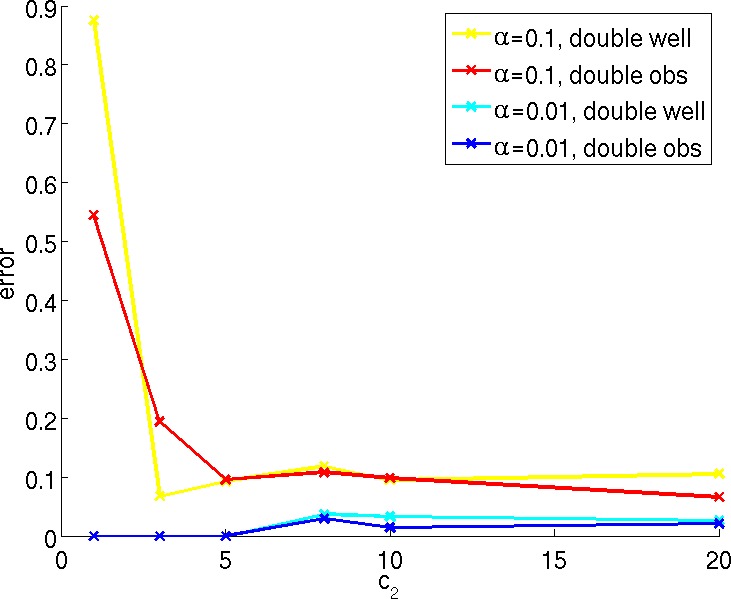}
\caption{Errors (averaged over many realisations of the noise) for both potentials at different levels of blurring and $\gamma=0.2$  with different values of $c_2$.
}
\label{fig:wfac}
\end{figure}

\subsection{Choice of iterative parameter}
\label{sec:it}

The discrete iterative methods of Section \ref{sec:discmthds} have values $\bar{\rho}$ independent of $h$ such that for all $\rho>\bar{\rho}$ the iterates decrease in energy and converge in some sense. For example, a possible $\bar{\rho}$ for the iterative method of Example \ref{ex:do22} applied to the problem in Section \ref{sec:speed} is $\max\{ \frac{\sigma_2}{\eps},C_s^2 \}=0.999$, where we use the Poincar\'e constant $1/\pi$. However in practice we observe that the iterates of this method decrease in energy and converge for $\rho \geq 0.833$. It is advantageous to take $\rho$ small, as this results in fewer iterations and uses less total computational effort. So to maximise speed we experimentally determine a value of $\rho$ which is as small as possible while still reliably giving a decrease in energy and convergence of iterates. This approach also works for the iterative method of Remark \ref{rem:dwit} for the double well potential, which lies outside of our framework. So for the speed comparison in Section \ref{sec:speed} we use $\rho = 0.833$ for the smooth double well potential and $\rho = 0.588$ for the double obstacle potential. In the rest of the numerics, where speed is less of a concern, $\rho$ is taken large (and larger than $\bar{\rho}$ if it is known) to ensure we get the expected behaviour of the iterative methods.

\subsection{Choice of stopping criterion}
\label{sec:stop}

We will never quite reach the steady state of the iterative method, so a decision needs to be made about when we are sufficiently close. For this purpose we use the stopping criterion introduced in Section \ref{sec:algo} which terminates the algorithms when the $L^2$ norm of the difference between consecutive iterations is less than TOL. 

Mostly we take TOL small so that we are effectively finding the exact steady state, but for the comparison of speed in Section \ref{sec:speed} we need to avoid unnecessary iterations. Figure \ref{fig:stop} suggests about 170 iterations will take us quite close to the steady state for the problem under consideration. This corresponds to taking TOL=$3e-4$ for the smooth double well and TOL=$3.5e-4$ for the double obstacle, and we use these values for the `accurate recovery'.

In practice we just want a sufficiently accurate recovery as quickly as possible. Our feeling is that the binary function is usually sufficiently accurately recovered once the error is below $0.1$. At this stage the correct number of bars have formed and the locations are probably known well enough (e.g. for a different algorithm to interpret the binary function as a barcode). We see in Figure \ref{fig:stop} that the smooth double well potential achieves this in around 11 iteration, which corresponds to TOL=$1.5e-2$. The double obstacle potential achieves this in around 9 iterations, which corresponds to TOL=$4e-2$. We take these values for the `rough recovery'.

\begin{figure}[ht]
\centering
\subfigure[Error.]{
\includegraphics[width=0.46\textwidth,natwidth=850,natheight=650]{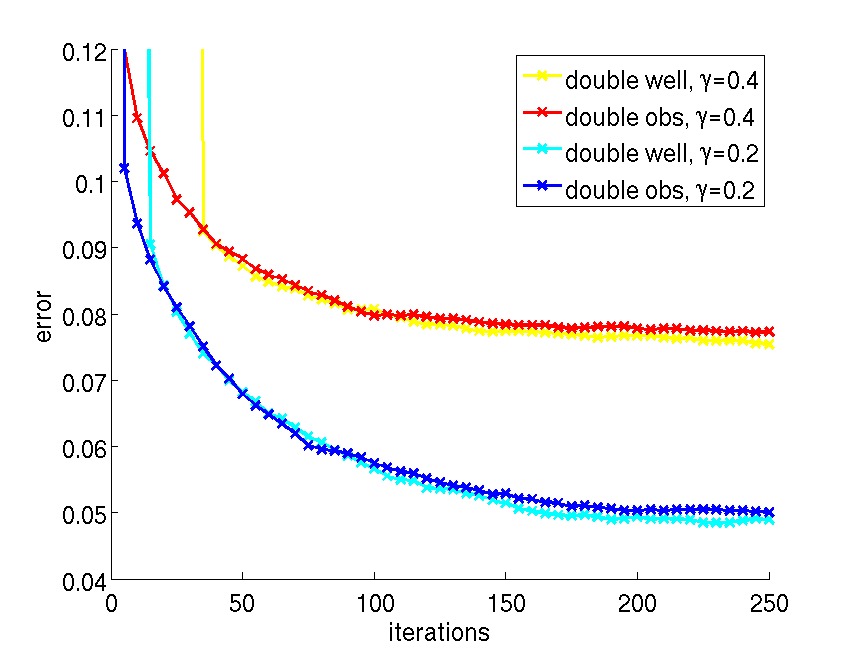}
\label{fig:err}
}
\subfigure[TOL.]{
\includegraphics[width=0.46\textwidth,natwidth=850,natheight=650]{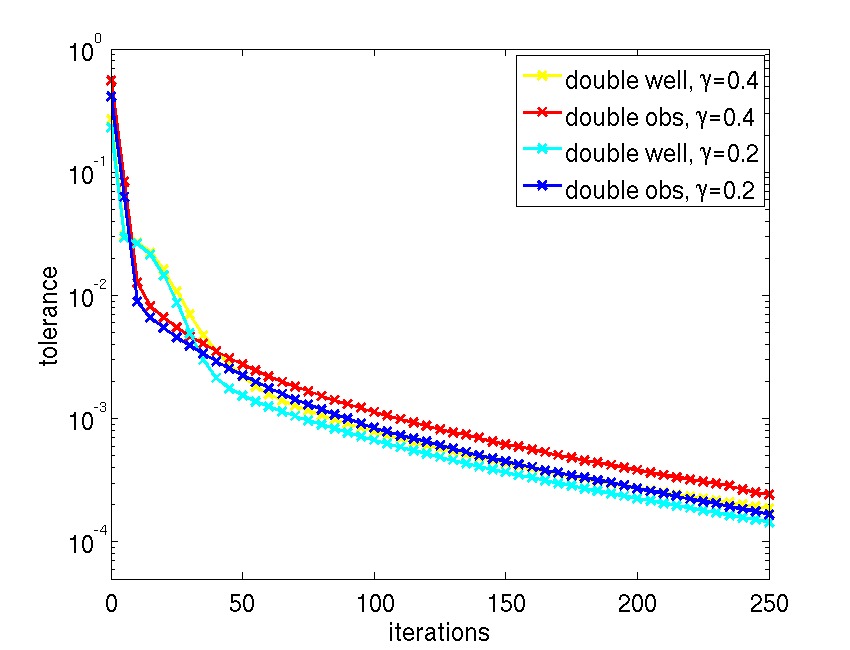}
\label{fig:tol}
}
\caption{The error (averaged over many realisations of the noise) after a given number of iterations for both potentials for the problem of Section \ref{sec:speed}.}
\label{fig:stop}
\end{figure}

\section*{Acknowledgements}
We are grateful to Carsten Gr\"{a}ser for sharing his Dune-Solvers code for the TNNMG method.

\bibliographystyle{alphaabbr}
\bibliography{paper}

\end{document}